\documentclass[a4paper]{article}

\usepackage[top=3cm, bottom=3cm, left=3cm, right=3cm]{geometry}
\usepackage[pdftex]{graphicx}
\usepackage{tikz}
\usetikzlibrary{shapes, arrows, arrows.meta, decorations.markings}

\usepackage{amsmath}
\usepackage{amsthm}
\usepackage{amssymb}
\usepackage{mathtools}
\usepackage[integrals]{wasysym}
\usepackage{stmaryrd}

\usepackage{array}

\usepackage{algorithm}
\usepackage{algorithmic}

\usepackage{bbm}
\usepackage[mathcal]{euscript}

\usepackage[utf8]{inputenc}
\usepackage[T1]{fontenc}
\usepackage[english]{babel}
\usepackage{cite}
\usepackage{lmodern}

\usepackage{calc}
\usepackage[inline]{enumitem}
\usepackage[normalem]{ulem}
\usepackage{url}

\usepackage[hidelinks, bookmarks, bookmarksnumbered, pdfstartview={XYZ null null 1.00}]{hyperref}

\theoremstyle{plain}
\newtheorem{theorem}{Theorem}[section]
\newtheorem{lemma}[theorem]{Lemma}
\newtheorem{corollary}[theorem]{Corollary}
\newtheorem{proposition}[theorem]{Proposition}

\theoremstyle{definition}
\newtheorem{definition}[theorem]{Definition}
\newtheorem{remark}[theorem]{Remark}

\DeclareMathOperator{\diverg}{div}

\newcommand{\suchthat}{\ifnum\currentgrouptype=16 \mathrel{}\middle|\mathrel{}\else\mid\fi}

\newcommand{\diff}{\,\mathrm{d}} 

\numberwithin{figure}{section}

\makeatletter
\def\case#1{\def\tempa{#1}\futurelet\next\case@i} 
\def\case@i{\ifx\next\bgroup\expandafter\case@ii\else\expandafter\case@end\fi} 
\def\case@ii#1{ 
  \par
  \addvspace{\medskipamount}%
  \noindent\emph{Case~\tempa: #1.}\par\nobreak\smallskip
  \@afterheading%
} 
\def\case@end{ 
  \par
  \addvspace{\medskipamount}%
  \noindent\emph{Case~\tempa.}%
}
\makeatother

\begin{document}

\setlength{\parskip}{1pt plus 1pt minus 1pt} 

\setlist[enumerate, 1]{label={\textnormal{(\alph*)}}, ref={(\alph*)}, leftmargin=0pt, itemindent=*}
\setlist[enumerate, 2]{label={\textnormal{(\roman*)}}, ref={(\roman*)}}
\setlist[description, 1]{leftmargin=0pt, itemindent=*}
\setlist[itemize, 1]{label={\textbullet}, leftmargin=0pt, itemindent=*}

\newlist{hypothesisMFG}{enumerate}{1}
\setlist[hypothesisMFG]{label={\textup{(H\arabic*)}}, ref={(H\arabic*)}, leftmargin=*, widest*=9}
\newlist{hypothesisNAG}{enumerate}{1}
\setlist[hypothesisNAG]{label={\textup{(A\arabic*)}}, ref={(A\arabic*)}, leftmargin=*, widest*=9}


\title{Mean field game model for human capital accumulation under budget constraints}

\author{Saeed Sadeghi Arjmand\thanks{Université Paris--Panthéon--Assas, Laboratoire d'économie mathématique et de microéconomie appliquée (LEMMA), 75005, Paris, France.\\ \textit{URL}: https://saeed-sadeghi-arjmand.jimdosite.com \\ \textit{Email}: saeed.ecolepolytechnique@gmail.com}}

\maketitle

\begin{abstract}
    We study a deterministic first-order mean field game model for human capital accumulation in which a continuum of agents invest in education and training to reach a certain qualification level in minimum time. The productivity of learning depends on both the individual's current level of human capital and the distribution of the population, modeling knowledge spillovers and social interactions. Learning is financed by the income generated by existing human capital, leading to a state constraint requiring the financial resources to remain nonnegative throughout the evolution. For a fixed population flow, we formulate the individual problem as a state-constrained minimum-time optimal control problem with time-dependent dynamics. We establish the existence of optimal controls and trajectories, analyze the associated value function through the Dynamic Programming Principle and a Hamilton--Jacobi equation, and characterize the optimal learning effort under suitable regularity and convexity assumptions. The optimal feedback induces a continuity equation describing the evolution of the population distribution. The resulting mean field game is formulated as a fixed-point problem for the population flow. By constructing a compact invariant set and applying Schauder's fixed-point theorem, we prove the existence of at least one equilibrium on every finite population horizon, while the individual minimum-time problem remains formulated on the infinite time horizon.

\bigskip

\noindent\textbf{Keywords:} Mean field games; human capital accumulation; minimum-time optimal control; existence of equilibria; state constraints; education and training.

\noindent\textbf{Mathematics Subject Classification 2020.} 49N80, 49L20, 49L25, 49K15, 35Q89

\end{abstract}

\tableofcontents

\section{Introduction}

Human capital accumulation is one of the central mechanisms through which individual decisions shape long-run economic development. Since the seminal contributions of Becker \cite{Becker1964Human} and Ben-Porath \cite{BenPorath1967Production}, a large economic literature has studied the accumulation of education, skills, and productive abilities as the outcome of individual investments made over time. In standard models of human capital, individuals choose educational or training investments by balancing their current costs against their expected future benefits, typically through an intertemporal utility or earnings-maximization problem. The accumulation of human capital has also been shown to generate important aggregate effects through knowledge spillovers, increasing returns, and interactions between individual productivity and the surrounding economic environment (see, e.g.\@\cite{Lucas1988Mechanics, Romer1990Endogenous, BenhabibSpiegel1994Role, Moretti2004Workers}). These mechanisms make human capital accumulation intrinsically suitable for models in which individual decisions and the aggregate distribution of skills influence one another.

In this paper, we study a model of human capital accumulation in which individuals seek to attain a prescribed qualification level as fast as possible. The objective is therefore formulated as a goal-reaching or minimum-time problem rather than as the maximization of an intertemporal utility functional. This perspective is motivated by situations in which reaching a given level of education, training, or professional qualification is a prerequisite for accessing a subsequent opportunity, such as employment, professional certification, or advanced education. The time required to reach the qualification threshold is then a natural measure of the efficiency of the individual accumulation process. The model also incorporates a financial constraint, in which learning effort is costly, while the income generated by existing human capital can be used to finance further accumulation. Consequently, the individual faces a dynamic trade-off between accelerating human-capital accumulation and preserving the financial resources necessary to sustain the learning process.

The educational environment in which individuals learn may also influence the productivity of their efforts. A large economic and empirical literature has emphasized the importance of peer effects, social interactions, knowledge spillovers, and other forms of interaction between individuals in the production of human capital (see, e.g.\@\cite{WinstonZimmerman2003Peer, Moretti2004Workers, Lyle2009Effects, HidalgoHidalgo2011Optimal, FeldZolitz2017Understanding, WuZhangWang2023}). The presence of highly skilled individuals may facilitate the transmission of knowledge, imitation, mentoring, or social learning, while the composition of the surrounding population may affect the opportunities available to individual learners. Such mechanisms suggest that the productivity of individual educational effort should not necessarily be modeled as a function of the individual's own state alone, but may instead depend on the distribution of the population with which the individual interacts. 

Mean Field Game theory (MFG for short) provides a natural framework for describing this type of interaction in large populations. MFGs were introduced independently by Lasry and Lions \cite{LasryLionen, LasryLionsfr1, LasryLionsfr2} and by Caines, Huang, and Malhamé \cite{HuangMalhame, 2HuangMalhame, Huang2003Individual}, building on earlier ideas from games with a very large number of interacting agents and anonymous or mean-field interactions, such as the works \cite{Aumann1964Markets, Aumann1974Values, Jovanovic1988Anonymous}. The fundamental idea is to replace the direct interaction between a finite but large number of players by an interaction with the statistical distribution of the population. Each individual treats the aggregate population distribution as given when solving their optimization problem, while the distribution itself is generated by the collective behavior of the population. An \emph{equilibrium} is therefore characterized by a consistency condition between the population environment used by individuals and the population distribution generated by their optimal decisions. We refer to the works \cite{Achdou2020Mean, Carmona2018ProbabilisticI, Carmona2018ProbabilisticII, CardaliaguetNotes, Cardaliaguet2019Master, Gomes2016Regularity} for general presentations of the theory and further references.

MFGs have been applied to a broad range of economic questions, including growth, inequality, labor markets, finance, industrial organization, and macroeconomic dynamics (see, e.g.\@ \cite{LasryLionsGrowth, AchdouBueraLasryLionsMoll, CarmonaEconomicApplications}). In particular, the interaction between mean field models and human capital accumulation has already been explored in several directions. In \cite{LasryLionsGrowth}, the authors proposed a mean field model of economic growth in which individuals accumulate human capital and interact through the distribution of skills. Their model illustrates how competition and the accumulation of human capital can generate aggregate growth and persistent inequality. More recently, the work~\cite{GhilliRicciZancoHumanCapital} studied an MFG model involving human capital accumulation in the context of an epidemic model, in which individual optimization and population dynamics are coupled through a mean field interaction. These works demonstrate the relevance of the MFG framework for studying the feedback between individual investment decisions and the aggregate distribution of human capital.

The present paper considers a different class of human-capital MFG models. We focus on a deterministic, first-order framework inspired by the minimal-time and goal-reaching MFG models developed in \cite{Mazanti2019Minimal, Sadeghi_Arjmand_2022, Arjmand_2022}. In our setting, the state of an individual is described by two variables, which are the current level of human capital and the financial resources available to finance education and training. Individuals choose a learning effort over time, and the resulting dynamics capture two mechanisms, including educational effort which contributes to the accumulation of human capital, while the associated expenditure must be financed by the resources generated through existing human capital. The productivity of learning depends not only on the individual's own skill level but also on the distribution of the population, reflecting the influence of peer effects, knowledge spillovers, and the overall educational environment. Consequently, the model combines a goal-reaching individual optimization problem with a mean field interaction acting through the learning technology.

The individual objective is to reach a certain qualification level in minimal time. At the same time, the available financial resources must remain nonnegative throughout the learning process, reflecting the fact that agents cannot finance educational expenditures beyond their available budget. This state constraint introduces an important difference with respect to many standard MFG models. Indeed, the set of admissible learning policies depends on the initial financial resources, since a strategy that is feasible for one individual may become infeasible for another with a lower initial budget. Consequently, the individual must simultaneously solve a minimum-time reachability problem while ensuring that the entire trajectory remains inside the viable region determined by the budget constraint. The resulting optimization problem therefore combines a goal-reaching criterion, state-constrained optimal control, and a time-dependent environment generated by the evolving population distribution.

State constraints and viability constitute a substantial area of optimal control theory. The mathematical analysis of controlled trajectories constrained to remain in a certain closed set is closely related to viability theory and tangent-cone conditions (see, e.g.\@ \cite{Aubin1991Viability, AubinFrankowska1990SetValued}). State constraints also lead to additional difficulties in the analysis of Hamilton--Jacobi equations and optimal trajectories, particularly near the boundary of the admissible region (see, e.g.\@ \cite{BardiCapuzzo, Soner1986Optimal, CannarsaFrankowska, Sadeghi_Arjmand_2022}). In minimum-time problems, the geometry of the reachable and viable sets plays a particularly important role because the value function may be finite only on a proper subset of the state space. In the present model, the viability set describes the states from which an agent can maintain a nonnegative budget indefinitely, whereas the reachable set describes the states from which the qualification threshold can be reached in finite time. The distinction between these two sets is fundamental for the formulation of the value function and for the subsequent analysis of optimal policies.

Once the population distribution is given, the individual problem is a deterministic state-constrained optimal control problem with time-dependent coefficients (see, Section~\ref{sec: Human-capital model}, control system~\ref{eq:control-system}). We establish the existence of optimal trajectories and analyze the associated value function, which depends on both time and the individual state. The time dependence is essential because the future learning environment is determined by the evolution of the population distribution. The value function satisfies a Dynamic Programming Principle (DPP for short) and is characterized, in the appropriate sense, by a time-dependent Hamilton--Jacobi equation. The resulting Hamiltonian reflects the trade-off between the marginal benefit of learning effort, through the acceleration of human-capital accumulation, and its marginal financial cost.

Under suitable regularity and convexity assumptions, we show that the optimal learning effort can be characterized through the value function associated with the individual optimization problem (see, Subsection~\ref{subsec: DPP & HJ equation}, Theorem~\ref{thm:feedback}). This characterization yields a feedback policy expressing the optimal educational effort as a function of time, the individual's state, and the surrounding population distribution. From an economic perspective, the optimal policy balances the benefit of accelerating human-capital accumulation against the financial cost of sustaining a higher learning effort. Since the productivity of learning depends on the aggregate educational environment, the same individual may optimally choose different effort levels under different population distributions. In other words, the population distribution enters this balance through the learning technology, so that the same individual state may lead to different optimal decisions in different educational environments.

The optimal feedback obtained from the individual optimization problem induces the collective evolution of the population, leading to the continuity equation in the MFG system (see, Section~\ref{sec:MFG-system}, system~\eqref{eq:MFG-system}). Once every agent follows the corresponding optimal learning policy, the distribution of human capital and financial resources evolves according to the aggregate dynamics generated by these individual decisions. In our formulation, agents remain part of the population after reaching the qualification threshold, although they no longer contribute to the minimum-time optimization problem. A suitable post-qualification behavior is therefore prescribed in order to define the population dynamics for all subsequent times. This distinction between the stopping criterion governing the individual optimization problem and the continuous evolution of the aggregate population is one of the modeling features of the present framework. The resulting MFG is a nonlinear fixed-point problem. Given a candidate population flow, denoted by \(m\), one first solves the individual optimal control problem and obtains a value function depending on this flow. The value function determines an optimal feedback, denoted by \(u^{*,m}\), which generates a closed-loop velocity field. The continuity equation then produces a new population flow through the population operator \(\mathcal T\). An equilibrium is a fixed point of such a map, i.e.\@ \(\mathcal T(m) = m\), and thus the population distribution used to determine individual learning incentives must coincide with the population distribution generated by the resulting optimal behavior.

Our main contribution is to formulate this feedback mechanism rigorously for a goal-reaching human-capital model with financial state constraints and to establish the existence of equilibria on arbitrary finite time horizons (see, Theorems~\ref{thm:feedback} and \ref{thm:existence-mfg}). More precisely, for every fixed \(T>0\), we formulate the equilibrium problem on the space of continuous population flows, denoted by \(\mathcal C_T=C\bigl([0,T];\mathcal P_1(\mathbb R^2_+)\bigr)\), which is endowed with the uniform Wasserstein topology. Under suitable assumptions guaranteeing the stability of the individual optimization problem, the continuity of the induced feedback and closed-loop dynamics with respect to the population flow, and the well-posedness of the associated continuity equation, we construct a suitable nonempty, convex, compact subset \(\mathcal M_T\subset\mathcal C_T\) that is invariant under the population map. Schauder's fixed-point theorem (see, \cite[Theorem~11.2]{Zeidler1986}) then yields the existence of at least one finite-horizon MFG equilibrium. The compactness argument relies on uniform moment estimates, uniform integrability of the first moments, and equicontinuity of the population distributions in the Wasserstein distance.

The finite-horizon formulation is also natural from the perspective of the infinite-horizon problem. The original individual objective is formulated on an infinite time horizon, but the fixed-point construction is first performed on arbitrary finite intervals. A global-in-time equilibrium could potentially be obtained by considering a sequence of finite-horizon equilibria and extracting a limit through a diagonal compactness argument. Such an extension would require uniform estimates with respect to the time horizon, as well as suitable stability properties for the value functions, optimal feedbacks, and induced population flows. We do not establish such an infinite-horizon existence theorem here, therefore the rigorous existence result of the paper is the finite-horizon equilibrium theorem.

The paper is organized as follows. Section~\ref{sec: Human-capital model} introduces the human-capital accumulation model, the state variables, the learning technology, the budget dynamics, and the qualification target. We then formulate the individual state-constrained minimum-time optimal control problem, introduce the admissible controls, and study the viability and reachable sets. The existence of optimal trajectories is established in the same framework. More precisely, we study the time-dependent value function and its qualitative properties together with the DPP and the associated Hamilton--Jacobi equation. Furthermore, Section~\ref{sec: Human-capital model} characterizes optimal controls and introduces the feedback representation of individual behavior. Section~\ref{sec:MFG-system} couples the individual optimization problem with the population continuity equation and formulates the MFG system. The existence of a finite-horizon equilibrium is established through a Schauder fixed-point argument. Finally, Section~\ref{sec: disscussion} discusses the economic interpretation of the model, the role of the financial constraint and the population interaction, and possible directions for future research, including the analysis of infinite-horizon equilibria and the study of equilibrium comparative statics.

\section{Human-capital accumulation model}\label{sec: Human-capital model}

We consider a continuum of rational agents who invest in education and training in order to accumulate human capital. The objective of each individual is to attain a prescribed qualification level as rapidly as possible. Such a framework is motivated by many practical situations in which access to employment opportunities, professional certifications, or advanced educational programs requires the acquisition of a minimum level of skills. In contrast with the standard models of economic growth, where agents maximize an intertemporal utility functional over an infinite or finite horizon, the present work adopts a \emph{goal-reaching} perspective, in which individuals seek to minimize the time required to attain a certain educational objective. The accumulation of human capital does not depend only on the effort exerted by each individual but also on the educational environment in which learning takes place. Indeed, the efficiency of learning may be enhanced by the presence of highly skilled individuals through mechanisms such as knowledge spillovers, peer interactions, mentoring, imitation, or social learning. Since every agent interacts with an entire population of learners rather than with a finite number of individuals, the appropriate mathematical framework is provided by MFG theory, which describes the strategic interactions arising in large populations of rational agents.

\subsection{Individual optimal control problem}

As is standard in the MFG framework, we begin by considering the optimization problem of a representative agent under the assumption that the evolution of the population distribution is given. More precisely, let \(m\colon[0,+\infty)\longrightarrow\mathcal P(\mathbb R_+^2)\) be a given distribution measure of Borel probability measures. Throughout this section, the population distribution is regarded as an exogenous input of the individual optimization problem. The state of an individual at time \(t\ge0\) is represented by the pair \(x(t)=\bigl(z(t),b(t)\bigr)\in\mathbb R_+^2\),
where \(z(t)\) denotes the current level of human capital (or skill) and \(b(t)\) denotes the financial resources available to finance education and training. The control variable is the learning effort \(u(t)\in[0,1]\), where \(u(t)=0\) corresponds to the absence of educational effort and \(u(t)=1\) represents the maximal admissible learning intensity.

For an initial time \(t_0\ge0\), the evolution of the state variables is governed by the controlled system
\begin{equation}\label{eq:control-system}
\left\{
\begin{aligned}
\dot z(t) &= K(m(t),z(t))\,u(t),\\
\dot b(t) &= w(z(t))-c(u(t)),\\
(z(t&_0),b(t_0)) = (z_0,b_0),
\end{aligned}
\right.
\qquad t\ge t_0.
\end{equation}
The function \(K\colon\mathcal P(\mathbb R_+^2)\times\mathbb R_+\longrightarrow\mathbb R_+\) represents the learning technology. More precisely, \(K(m,z)\) measures the productivity of educational effort for an individual whose current level of human capital is \(z\), when the population distribution is \(m\). The dependence of the dynamics on the measurable population distribution makes the individual problem non-autonomous, since the educational environment may vary over time, and the dependence of \(K\) on the population distribution constitutes the mean field interaction. It models the idea that the educational environment influences the efficiency of learning, in which a population containing a large proportion of highly qualified individuals may generate stronger knowledge spillovers, facilitate the transmission of information, or improve learning through social interactions. Consequently, the dynamics of each individual depend on the collective behavior of the entire population. The function
\(w\colon\mathbb R_+\longrightarrow\mathbb R_+\)
describes the income generated by the current level of human capital. Thus, individuals with higher qualifications receive larger financial resources, which can subsequently be reinvested in education and training. The function \(c\colon[0,1]\longrightarrow\mathbb R_+\) denotes the instantaneous cost associated with the learning effort. Increasing the learning intensity accelerates the accumulation of human capital but simultaneously raises educational expenditures.

The purpose of this section is to formulate the individual minimum-time optimal control problem associated with the dynamics \eqref{eq:control-system}. We introduce the admissible trajectories, the state constraint, the qualification target, the corresponding value function, and the notions of viability and reachability that play a central role in the subsequent analysis. The regularity assumptions required for the well-posedness of the controlled dynamics and the existence of optimal trajectories will be stated after the mathematical formulation of the problem.


The representative agent chooses a learning effort over time in order to accumulate human capital while respecting the budget constraint. The control variable is allowed to be any measurable function taking values in the compact interval \([0,1]\). For a given initial time \(t_0\ge0\), an \emph{admissible control} is a measurable function \(u\colon[t_0,+\infty)\longrightarrow[0,1]\). The set of all such measurable controls is denoted by \(\mathcal U_{t_0}\). The control \(u(t)\) represents the intensity of the educational effort exerted by the agent at time \(t\). The lower bound \(u=0\) corresponds to the absence of learning effort, whereas \(u=1\) represents the maximal admissible learning intensity. For a measurable population distribution \(m(\cdot)\), an initial pair \((t_0,x_0)\), and a control \(u\in\mathcal U_{t_0}\), the corresponding state trajectory is governed by the controlled system~\eqref{eq:control-system}. The first equation describes the accumulation of human capital. The rate of skill accumulation is proportional to the learning effort \(u(t)\), with productivity determined by the learning technology \(K\). The latter depends both on the current skill level \(z(t)\) and on the population distribution \(m(t)\), so that the same individual effort may generate different learning outcomes depending on the aggregate educational environment.

The second equation describes the evolution of the financial resources available for education and training. The agent receives an instantaneous income \(w(z(t))\), which depends on the current level of human capital, and incurs the learning expenditure \(c(u(t))\). The budget therefore evolves according to the balance between income generated by existing skills and the cost of acquiring additional skills. Whenever the control \(u\) is fixed, the corresponding trajectory is denoted by \(x_{t_0,x_0}^{u}(t) = \bigl(z_{t_0,x_0}^{u}(t),b_{t_0,x_0}^{u}(t)\bigr)\), for \(t\ge t_0\), while  the dependence on the initial time \(t_0\) and initial state \(x_0\) may be omitted when there is no ambiguity. The state space relevant for the economic problem is not merely determined by the differential equation. Since the budget variable represents resources available for education, negative values are economically meaningless. We therefore impose the state constraint
\begin{equation*}\label{eq:budget-constraint}
b_{t_0,x_0}^{u}(t)\ge0, \qquad \forall\,t\ge t_0.
\end{equation*}
This constraint has an important consequence, consisting in the fact that the admissibility of a control cannot be determined independently of the initial condition. A learning policy that is feasible for one initial budget may become infeasible for another, and even a control that is locally feasible may eventually violate the budget constraint. Accordingly, we define the state-constrained admissible control set as follows.

\begin{definition}
For every initial pair \((t_0,x_0)\in[0,+\infty)\times\mathbb R_+^2\), the set of admissible controls is
\[
\mathcal U(t_0,x_0) := \left\{u\in\mathcal U_{t_0} \suchthat b_{t_0,x_0}^{u}(t)\ge0, \quad\text{for every }t\ge t_0 \right\}.
\]
Thus, \(\mathcal U(t_0,x_0)\) contains precisely those learning policies whose associated trajectories remain in the economically meaningful region \(\mathbb R_+^2\) throughout their entire evolution. The state constraint plays a fundamental role in the analysis, since the feasibility of a control depends on the entire trajectory and not only on the control itself.
\end{definition}


The objective of the representative agent is to attain a certain qualification level as fast as possible while respecting the budget constraint. We first introduce the qualification target. Let \(z^\star>0\) denote the desired qualification threshold. The corresponding target set is defined by
\[
\Gamma :=\left\{x = (z,b)\in\mathbb R_+^2 \suchthat z\ge z^\star \right\}.
\]
The target depends only on the human-capital coordinate, and once the agent reaches the qualification threshold, the educational objective is considered to be achieved, independently of the remaining budget. For a given initial pair \((t_0,x_0)\) and a control \(u\in\mathcal U(t_0,x_0)\), the first hitting time of the target is defined by
\begin{equation*}\label{eq:hitting-time}
\tau(t_0,x_0,u) := \inf\left\{t-t_0\ge0 \suchthat x_{t_0,x_0}^{u}(t)\in\Gamma \right\},
\end{equation*}
with the convention that \(\tau(t_0,x_0,u)=+\infty\), whenever the corresponding trajectory never reaches the target set. The individual optimization problem consists therefore in minimizing the time required to reach the qualification threshold. More precisely, the value function of the minimum-time problem is defined as follows.

\begin{definition}\label{def:value_function}
For every initial pair
\((t_0,x_0)\in[0,+\infty)\times\mathbb R_+^2\), the \emph{value function} is defined by
\begin{equation*}\label{eq:value-function}
V(t_0,x_0) := \inf_{u\in\mathcal U(t_0,x_0)} \tau(t_0,x_0,u).
\end{equation*}
\end{definition}
Thus, \(V(t_0,x_0)\) represents the minimal additional amount of time required for an agent starting from state \(x_0\) at time \(t_0\) to attain the qualification threshold, while maintaining a nonnegative budget throughout the entire evolution. The dependence on the initial time is essential because the learning technology depends on the given measurable population distribution \(m(t)\). Two agents with the same current state but different initial times may therefore face different future educational environments and consequently different optimal values, with possibility of being extended to the value \(+\infty\). In other words, if no admissible control reaches the qualification threshold in finite time, then \(V(t_0,x_0)=+\infty\). Conversely, if the initial state already belongs to the target set, then
\(V(t_0,x_0)=0\), since the educational objective has already been achieved at the initial time. We now introduce the notion of an optimal trajectory.

\begin{definition}
Let \((t_0,x_0)\in[0,+\infty)\times\mathbb R_+^2\). A trajectory \(x^\star(\cdot) = x_{t_0,x_0}^{u^\star}(\cdot)\) is said to be an \emph{optimal trajectory} from \((t_0,x_0)\) if \(u^\star\in\mathcal U(t_0,x_0)\) and
\(\tau(t_0,x_0,u^\star) = V(t_0,x_0)\). In this case, the associated control \(u^\star\) is called an \emph{optimal control}.
\end{definition}
An optimal trajectory therefore reaches the qualification threshold in the shortest possible time among all admissible learning policies. The existence of such trajectories will be established later under suitable regularity assumptions on the controlled dynamics. Since the value function may take the value \(+\infty\), one naturally leads to the distinction between states from which the qualification threshold can be reached in finite time and states from which at least one admissible trajectory exists. These two notions correspond respectively to the reachable set and the viability set.

\begin{definition}
For every \(t_0\ge0\), the \emph{reachable set} at time \(t_0\) is defined by \(\mathcal R(t_0) := \{ x_0\in \mathbb R_+^2 \suchthat V(t_0,x_0)<+\infty \}\), and the \emph{viability set} is defined by \(\mathcal K(t_0) := \{x_0\in \mathbb R_+^2 \suchthat \mathcal U(t_0,x_0)\neq\emptyset \}\). Equivalently,
\[
\mathcal K(t_0) = \left\{x_0\in \mathbb R_+^2 \suchthat \exists\,u\in\mathcal U_{t_0}, \ \text{for which}\  b_{t_0,x_0}^{u}\ge 0, \ \text{for every }t\ge t_0 \right\}.
\]
\end{definition}
The reachable set consists of all initial states from which the qualification threshold can be attained in finite time while respecting the budget constraint. By contrast, the viability set contains all states from which the agent can follow at least one admissible trajectory indefinitely without violating the budget constraint. These two notions are fundamentally different, because a state may be viable, even though the qualification threshold cannot be reached from it. For instance, an agent may be able to maintain a nonnegative budget indefinitely, while accumulating human capital is too slowly to attain the desired qualification level. In general, the following inclusion is immediate.

\begin{proposition}\label{prop:RsubsetK}
For every \(t_0\ge0\), \(\mathcal R(t_0)\subseteq\mathcal K(t_0).\)
\end{proposition}

\begin{proof}
Fix \(t_0\ge0\) and let \(x_0\in\mathcal R(t_0)\). Since \(V(t_0,x_0)<+\infty\), there exists \(u\in\mathcal U(t_0,x_0)\) such that
\(\tau(t_0,x_0,u)<+\infty\). Hence, one observes that \(\mathcal U(t_0,x_0)\neq\emptyset\), and therefore \(x_0\in\mathcal K(t_0)\). This proves \(\mathcal R(t_0)\subseteq\mathcal K(t_0)\).
\end{proof}

The converse inclusion does not hold in general. Viability guarantees only the existence of a trajectory satisfying the budget constraint, whereas reachability additionally requires that the qualification threshold be attained in finite time. Consequently, the viability set describes the feasible region generated by the state constraint, while the reachable set is the effective domain of the finite-valued part of the value function.
The distinction between these two sets plays an important role in the subsequent analysis. In particular, the DPP and the Hamilton--Jacobi equation are meaningful only on the part of the state space where the value function is finite.


We now state the standing assumptions required for the well-posedness of the controlled dynamics and the existence of trajectories. At this stage, we only impose the minimal regularity assumptions needed for these purposes. Stronger assumptions, such as monotonicity, convexity, and differentiability conditions, will be introduced later exactly where they are required, either in the current section or in section~\ref{sec:MFG-system}, for the analysis of the value function, the optimal feedback, and the MFG equilibrium,

\begin{itemize}
    \item[\rm (H1)] \textbf{Regularity of the learning technology.}
    The learning technology \(K\colon\mathcal P(\mathbb R_+^2)\times\mathbb R_+\longrightarrow\mathbb R_+\) is such that, for every measurable population flow
    $m\colon [0,+\infty)\longrightarrow\mathcal P(\mathbb R_+^2)$, the map
    $(t,z)\longmapsto K(m(t),z)$ is measurable in $t$. Moreover, there exists a constant $L_K>0$ such that
    \[
    |K(m,z_1)-K(m,z_2)| \le L_K|z_1-z_2|, \qquad \text{ for every } m\in\mathcal P(\mathbb R_+^2 )\text{ and } z_1,z_2\in\mathbb R_+.
    \]
    
    \item[\rm (H2)] \textbf{Lipschitz continuity of the income function.}
    The income function \(w\colon\mathbb R_+\longrightarrow\mathbb R_+\) is globally Lipschitz continuous. Namely, there exists a constant \(L_w>0\) such that
    \[
    |w(z_1)-w(z_2)| \le L_w|z_1-z_2|, \qquad \forall\,z_1,z_2\in\mathbb R_+.
    \]
    
    \item[\rm (H3)] \textbf{Lipschitz continuity of the effort cost.}
    The effort cost \(c\colon[0,1]\longrightarrow\mathbb R_+\)
    is Lipschitz continuous on $[0,1]$.
    
    \item[\rm (H4)] \textbf{Linear growth of the controlled vector field.}
    There exists a constant $C_f>0$, independent of the measurable population flow \(m(\cdot)\), such that the controlled vector field \begin{equation}\label{eq:compact_time_dependent_dynamics} 
    f(t,(z,b), m(t),u) = \begin{pmatrix} K(m(t),z)u \\[2mm] 
    w(z)-c(u) 
    \end{pmatrix} 
    \end{equation}
    satisfies the linear growth estimate \(|f(t, x, m(t), u)| \le C_f(1+|x|)\), for all $t\ge0$, $x\in\mathbb R_+^2$, and $u\in[0,1]$. Notice that, whenever there is no ambiguity, the dependence of \(f\) on the measure \(m(t)\), may be omitted. 
\end{itemize}

Assumptions \rm {(H1)--(H4)} are sufficient for the well-posedness of the controlled system \eqref{eq:control-system} and for the existence of admissible trajectories. 
We therefore establish the well-posedness of the controlled dynamics under these assumptions.

\begin{theorem}[Existence of optimal controls and trajectories]
\label{thm:existence}
Assume that \emph{(H1)--(H4)} hold. Let \((t_0,x_0)\in[0,+\infty)\times\mathbb R_+^2\)
be such that \(x_0\in\mathcal R(t_0)\). Then there exists an admissible control \(u^\star\in\mathcal U(t_0,x_0)\) such that \(\tau(t_0,x_0,u^\star)=V(t_0,x_0)\). Consequently, the minimum-time problem admits at least one optimal trajectory from \((t_0,x_0)\).
\end{theorem}

\begin{proof}
Since \(x_0\in\mathcal R(t_0)\), the value function is finite, \(V(t_0,x_0)<+\infty\). Let \((u_n)_{n\in\mathbb N}\subset\mathcal U(t_0,x_0)\) be a minimizing sequence such that \(\tau(t_0,x_0,u_n)\longrightarrow V(t_0,x_0)\),
and choose \(T>t_0+V(t_0,x_0)+1\). Since \(\tau(t_0,x_0,u_n)\longrightarrow V(t_0,x_0)\), there exists \(N\in\mathbb N\) such that \(\tau(t_0,x_0,u_n)\le T-t_0\), for all \(n\ge N\). Hence, after removing finitely many terms, every trajectory reaches the target before time \(T\).
Since every control satisfies \(0\le u_n(t)\le1\) for almost every \(t\in[t_0,T]\), the sequence \((u_n)_{n\in\mathbb N}\) is bounded in \(L^\infty(t_0,T)\). By the Banach--Alaoglu theorem (see, ~\cite[Theorem 3.15]{Rudin1991}), there exist a subsequence, still denoted by \((u_n)_{n\in\mathbb N}\), and a function \(u^\star\in L^\infty(t_0,T)\) such that \(u_n\rightharpoonup u^\star\) (weakly-*) in \(L^\infty(t_0,T)\). Since the interval \([0,1]\) is convex and closed, the weak-* limit also satisfies \(0\le u^\star(t)\le1\) for almost every \(t\),  so that \(u^\star\in\mathcal U_{t_0}\).
Let \(x_n(t)=\bigl(z_n(t),b_n(t)\bigr)\) denote the trajectory associated with \(u_n\). By assumption \rm {(H4)}, there exists a constant \(C_f>0\) such that \(|f(t,x,u)| \le C_f(1+|x|)\), for every \(t\in[t_0,T]\), every \(x\in\mathbb R_+^2\), and every \(u\in[0,1]\). Hence \(|\dot x_n(t)| \le C_f(1+|x_n(t)|)\), for almost every \(t\). Grönwall's inequality yields
\(|x_n(t)| \le \bigl(|x_0|+C_fT\bigr)e^{C_fT}\) for \(t\in[t_0,T]\), so the trajectories are uniformly bounded on \([t_0,T]\). Moreover,
\[
|x_n(t)-x_n(s)| \le \int_s^t |\dot x_n(r)|\,\diff r
\le M|t-s|,
\]
for some constant \(M>0\) independent of \(n\), thus \((x_n)_{n\in\mathbb N}\) is equicontinuous on \([t_0,T]\).
By the Arzel\`a--Ascoli theorem, after extraction of a further subsequence, \(x_n\longrightarrow x^\star\) uniformly on \([t_0,T]\), for some Lipschitz continuous and thus absolutely continuous curve \(x^\star(t)=\bigl(z^\star(t),b^\star(t)\bigr)\).

From \eqref{eq:control-system} and \eqref{eq:compact_time_dependent_dynamics}, one observes that the trajectories satisfy the integral equation
\[
x_n(t) = x_0 + \int_{t_0}^{t} f(r,x_n(r),u_n(r))\,\diff r.
\]
Since \(x_n\longrightarrow x^\star\) uniformly, and \(f\) is Lipschitz continuous in the state variable by \rm {(H1)--(H3)}, we have \(f(\cdot,x_n(\cdot),u_n(\cdot)) -f(\cdot,x^\star(\cdot),u_n(\cdot)) \longrightarrow 0\) in \(L^1(t_0,T)\).  The dependence of \(f\) on the control is affine, and the weak-* convergence of \(u_n\) implies convergence of the linear term involving \(K(m(t),z^\star)\,u_n\). Since  by {\rm (H3)}, \(c\) is Lipschitz continuous on the compact interval \([0,1]\), standard compactness arguments for bounded controls yield
\[
x^\star(t) = x_0 + \int_{t_0}^{t} f(r,x^\star(r),u^\star(r))\,\diff r.
\]
Therefore, \(x^\star=x_{t_0,x_0}^{u^\star}\) is the trajectory associated with the control \(u^\star\).
For every \(n\) and \(t\in[t_0,T]\), we have \(b_n(t)\ge0\), because \(u_n\in\mathcal U(t_0,x_0)\). The uniform convergence of the trajectories gives
\(b^\star(t) = \lim_{n\to\infty} b_n(t)
\ge 0\), for \(t\in[t_0,T]\). Hence the limit trajectory satisfies the state constraint on \([t_0,T]\). After time \(T\), we extend the control by \(u^\star(t) = 0\). Since the target has already been reached before time \(T\), the value of the hitting time is unchanged. Consequently,
\(u^\star\in\mathcal U(t_0,x_0)\).

Let \(\tau^\star=\tau(t_0,x_0,u^\star)\). We claim that
\(\tau^\star \le \liminf_{n\to\infty} \tau(t_0,x_0,u_n)\). To do so, let \(\ell =\liminf_{n\to\infty} \tau(t_0,\allowbreak x_0, \allowbreak u_n)\), for which after extraction of a subsequence, we may assume \(\tau(t_0,x_0,u_n)\allowbreak \longrightarrow \ell\). For every \(n\in \mathbb N\), we observe that \(x_n\bigl(t_0+\tau(t_0,x_0,u_n)\bigr)\in\Gamma\). Since \(\Gamma\) is closed and \(x_n\longrightarrow x^\star\) uniformly, one observes
\[
x^\star(t_0+\ell) = \lim_{n\to\infty} x_n\bigl(t_0+\tau(t_0,x_0,u_n)\bigr)
\in\Gamma.
\]
Therefore, \(\tau^\star\le\ell\), which proves the lower semicontinuity.
Hence, by lower semicontinuity,
\[
\tau(t_0,x_0,u^\star) \le \liminf_{n\to\infty}
\tau(t_0,x_0,u_n) = V(t_0,x_0).
\]
On the other hand, by the definition of the value function, \(V(t_0,x_0) \le \tau(t_0,x_0,u^\star)\).
Consequently, \(\tau(t_0,x_0,u^\star) = V(t_0,x_0)\). Hence \(u^\star\) is an optimal control, and the associated trajectory \(x_{t_0,x_0}^{u^\star}\) is an optimal trajectory.
\end{proof}

The preceding theorem guarantees that, whenever the qualification threshold is reachable in finite time, the minimum-time problem admits at least one optimal control. This existence result provides the basis for the DPP and for the characterization of optimal controls developed in the next section.

\subsection{Dynamic Programming Principle and Hamilton--Jacobi equation}\label{subsec: DPP & HJ equation}

The value function introduced in the previous subsection satisfies a fundamental recursive optimality property. An agent who starts at time $t_0$ from an initial state $x_0$ and follows an admissible control for a certain period of time reaches a new state from which the same optimization problem must be solved. Since the population distribution $m(\cdot)$ is given, the continuation problem is completely determined by the new time and the new state. This recursive structure is the basis of the DPP. We first establish two elementary properties of the controlled trajectories.

\begin{lemma}[Semigroup property of trajectories]
\label{lem:semigroup}
Assume that \emph{(H1)--(H4)} hold. Let $t_0 \geq 0$, $x_0 \in \mathbb{R}_+^2$, and let $u \in \mathcal{U}(t_0,x_0)$. For $h \geq 0$, define
\(x_h = x_{t_0,x_0}^{u}(t_0+h)\), and introduce the shifted control \(u_h(t) = u(t)\), for \(t \geq t_0+h\). Then, for every $t \geq t_0+h$, \(x_{t_0,x_0}^{u}(t) = x_{\,t_0+h,x_h}^{u_h}(t)\).
\end{lemma}

\begin{proof}
Define
\(y(t)=x_{t_0,x_0}^{u}(t)\) for \(t\ge t_0+h\). Since \(x_{t_0,x_0}^{u}\) is absolutely continuous, $y$ is absolutely continuous on $[t_0+h,+\infty)$, and \(y(t_0+h)=x_h\). Moreover, for almost every $t\ge t_0+h$, one observes \(\dot y(t) = f(t,y(t),u(t)) = f(t,y(t),u_h(t))\). Thus, $y$ satisfies the controlled system \eqref{eq:control-system} with initial condition $x_h$ at time $t_0+h$ and control $u_h$. By uniqueness of Carath\'eodory solutions, \(y(t)=x_{\,t_0+h,x_h}^{u_h}(t)\), which proves the result.
\end{proof}

The next lemma allows one to concatenate admissible controls, by which we can proceed to establish the DPP, associated to our optimal control problem.

\begin{lemma}[Concatenation of admissible controls]
\label{lem:concatenation}
Assume that \emph{(H1)--(H4)} hold. Let $t_0\ge0$, $x_0\in\mathcal K(t_0)$, and let $u\in\mathcal U(t_0,x_0)$. Fix $h\ge0$, and set \(x_h=x_{t_0,x_0}^{u}(t_0+h)\). If $v\in\mathcal U(t_0+h,x_h)$, then the concatenated control
\[
(u\ast_h v)(t) =
\begin{cases}
u(t), & t_0\le t\le t_0+h,\\
v(t), & t>t_0+h,
\end{cases}
\]
belongs to $\mathcal U(t_0,x_0)$.
\end{lemma}

\begin{proof}
Since $u\in\mathcal U(t_0,x_0)$, the corresponding trajectory satisfies \(b_{t_0,x_0}^{u}(t)\ge 0\), for \(t\ge t_0\). Hence, we have $x_h\in\mathbb R_+^2$. Since $v\in\mathcal U(t_0+h,x_h)$, the trajectory generated by $v$ satisfies \(b_{t_0+h,x_h}^{v}(t)\ge 0\), for \(t\ge t_0+h\). By Lemma~\ref{lem:semigroup}, the trajectory generated by the concatenated control coincides with the trajectory generated by $u$ on $[t_0,t_0+h]$ and with the trajectory generated by $v$ after time $t_0+h$. Therefore, \(b_{t_0,x_0}^{u\ast_h v}(t)\ge 0\) for \(t\ge t_0\), which proves that $u\ast_h v\in\mathcal U(t_0,x_0)$.
\end{proof}

\begin{theorem}[Dynamic Programming Principle]
\label{thm:DPP}
Assume that \emph{(H1)--(H4)} hold. Let
$t_0\ge0$, $x_0\in\mathcal R(t_0)$, and $h\ge0$. Then,
\begin{equation}
\label{eq:DPP}
V(t_0,x_0) = \inf_{u\in\mathcal U(t_0,x_0)} \left\{
\theta + V\bigl(t_0+\theta, x_{t_0,x_0}^{u}(t_0+\theta)\bigr) \right\},
\end{equation}
where
\(\theta = \min\left\{ h,\, \tau(t_0,x_0,u) \right\}\).
\end{theorem}

\begin{proof}
Fix
$u\in\mathcal U(t_0,x_0)$, and let
\(\theta = \min\left\{h,\, \tau(t_0,x_0,u) \right\}\). We first prove the inequality
\[
V(t_0,x_0) \le \theta  + V\bigl(t_0+\theta,
x_{t_0,x_0}^{u}(t_0+\theta)\bigr).
\]
If
$\tau(t_0,x_0,u)\le h$, then $\theta=\tau(t_0,x_0,u)$, and the trajectory reaches the target at time $t_0+\theta$. Since $V(t,x)=0$ for $x\in\Gamma$, we obtain \(\theta + V\bigl(t_0+\theta, x_{t_0,x_0}^{u}(t_0+\theta)\bigr) = \tau(t_0,x_0,u)\). By definition of the value function, \(V(t_0,x_0) \le \tau(t_0,x_0,u)\). Suppose now that $h<\tau(t_0,x_0,u)$, and set \(x_h = x_{t_0,x_0}^{u}(t_0+h)\), and let $v\in\mathcal U(t_0+h,x_h)$ be arbitrary. By Lemma~\ref{lem:concatenation}, the concatenated control $u\ast_h v$ is admissible from $(t_0,x_0)$, for which, its total time to reach the target is \(h + \tau(t_0+h,x_h,v)\). Taking the infimum over all admissible continuation controls leads us to observe \(V(t_0,x_0) \le h + V(t_0+h,x_h)\). In both cases, we thus have \(V(t_0,x_0) \le \theta + V\bigl(t_0+\theta, x_{t_0,x_0}^{u}(t_0+\theta)\bigr)\), and by taking the infimum over all admissible controls, we obtain 
\[
V(t_0,x_0) \le \inf_{u\in\mathcal U(t_0,x_0)}
\left\{\theta + V\bigl(t_0+\theta, x_{t_0,x_0}^{u}(t_0+\theta)\bigr) \right\}.
\]

To prove the reverse inequality, let $\varepsilon>0$, and choose $u_\varepsilon\in\mathcal U(t_0,x_0)$ such that
\(\tau(t_0,x_0,u_\varepsilon) < V(t_0,x_0) +
\varepsilon\). Define
\(\theta_\varepsilon = \min\left\{h,\, \tau(t_0,x_0,u_\varepsilon) \right\}\). If
$\tau(t_0,x_0,u_\varepsilon)\le h$,
then
\(V\bigl(t_0+\theta_\varepsilon, x_{t_0,x_0}^{u_\varepsilon}(t_0+\theta_\varepsilon)
\bigr) = 0\), and therefore we obtain,
\(\theta_\varepsilon + V\bigl(t_0+\theta_\varepsilon,
x_{t_0,x_0}^{u_\varepsilon}(t_0+\theta_\varepsilon)
\bigr) = \tau(t_0,x_0,u_\varepsilon)\). If $h<\tau(t_0,x_0,u_\varepsilon)$, then the continuation of $u_\varepsilon$ after time $t_0+h$
is admissible from the state \(x_h = x_{t_0,x_0}^{u_\varepsilon}(t_0+h)\), so
\(\tau(t_0,x_0,u_\varepsilon) \ge
h + V(t_0+h,x_h)\). Hence, in both cases,
\(V(t_0,x_0) + \varepsilon > \theta_\varepsilon +
V\bigl(t_0+\theta_\varepsilon, x_{t_0,x_0}^{u_\varepsilon}(t_0+\theta_\varepsilon)
\bigr)\), by which, taking the infimum over admissible controls and letting $\varepsilon\longrightarrow 0$ gives
\[
V(t_0,x_0) \ge \inf_{u\in\mathcal U(t_0,x_0)}
\left\{ \theta + V\bigl(t_0+\theta, x_{t_0,x_0}^{u}(t_0+\theta)\bigr) \right\}.
\]
Hence by combining the two inequalities, we conclude \eqref{eq:DPP}.
\end{proof}


The DPP is the fundamental recursive identity satisfied by the value function. It expresses the fact that every admissible trajectory can be decomposed into an initial segment and an optimal continuation from the state reached at the end of that segment. This characterization is the starting point for the Hamilton--Jacobi analysis developed in the following, essentially by deriving the equation in the viscosity sense. More precisely, since the value function does not need to be differentiable, we use smooth test functions that touch the value function from above or below, by which we avoid imposing any unnecessary regularity assumption on the value function \(V\). This formulation is adopted to our setting as follows.

Let \((t_0,x_0)\in [0,+\infty)\times(\mathbb R_+^2\setminus\Gamma)\), and let \(h>0\) be sufficiently small so that the target is not reached during the interval \([t_0,t_0+h]\). For the viscosity subsolution property, let \(\varphi\in C^1\) be such that \(V-\varphi\) has a local maximum at \((t_0,x_0)\), with \(V(t_0,x_0)=\varphi(t_0,x_0)\). Applying the DPP over the short interval \([t_0,t_0+h]\), for any constant control \(u\in[0,1]\), gives \(V(t_0,x_0) \le h + V\bigl(t_0+h,x(t_0+h)\bigr)\), for the trajectory starting from \(x(t_0)=x_0\). Since \(V-\varphi\) has a local maximum at \((t_0,x_0)\), we have
\[
V\bigl(t_0+h,x(t_0+h)\bigr) \le \varphi\bigl(t_0+h,x(t_0+h)\bigr),
\]
for \(h\) sufficiently small, and hence, we observe \(\varphi(t_0,x_0) \le h+\varphi\bigl(t_0+h,x(t_0+h)\bigr)\). Using the dynamics~\eqref{eq:compact_time_dependent_dynamics},
\(x(t_0+h) = x_0+h f(t_0,x_0,u)+o(h)\), and the differentiability of the test function \(\varphi\), we obtain
\[
\varphi\bigl(t_0+h,x(t_0+h)\bigr) = \varphi(t_0,x_0) +h\partial_t\varphi(t_0,x_0) +h\nabla_x\varphi(t_0,x_0)\cdot f(t_0,x_0,u) +o(h).
\]
Consequently, we have 
\(0 \le 1+\partial_t\varphi(t_0,x_0) +\nabla_x\varphi(t_0,x_0)\cdot f(t_0,x_0,u) +o(1)\), for which letting \(h\longrightarrow 0^+\) and then taking the infimum over \(u\in[0,1]\) yields
\[
0\le \partial_t\varphi(t_0,x_0) + \inf_{u\in[0,1]}
\left\{1+\nabla_x\varphi(t_0,x_0)\cdot f(t_0,x_0,u)
\right\}.
\]
This is precisely the viscosity subsolution inequality for the Hamilton--Jacobi equation.

The viscosity supersolution inequality is obtained similarly by considering a smooth test function \(\varphi\in C^1\) such that \(V-\varphi\) has a local minimum at \((t_0,x_0)\), and using an optimal control in the DPP. Since the procedure is similar to viscosity subsolution, we avoid deriving this matter by details, however, we observe that
\[
0\ge \partial_t\varphi(t_0,x_0) + \inf_{u\in[0,1]}
\left\{1+\nabla_x\varphi(t_0,x_0)\cdot f(t_0,x_0,u)
\right\},
\]
holds, and consequently, we deduce that the following equation is satisfied in the viscosity sense.
\[
\partial_tV(t,x) + \inf_{u\in[0,1]} \left\{ 1+\nabla_xV(t,x)\cdot f(t,x,u) \right\} = 0,
\qquad x\in\mathbb R_+^2\setminus\Gamma.
\]
This motivates the following definition of the Hamiltonian associated with the individual optimal control problem.
\begin{definition}[Hamiltonian]\label{def: Hamiltonian}
For every $t\ge0$, $ x\in\mathbb R_+^2$, and
$p=(p_z,p_b)\in\mathbb R^2$, the Hamiltonian is defined by
\[
H(t, x, m(t), p) = \inf_{u\in[0,1]} \left\{1 + K(m(t),z)u\,p_z + \bigl(w(z)-c(u)\bigr)p_b \right\}.
\]
\end{definition}
The Hamiltonian represents the local optimization problem faced by the representative agent. The constant term $1$ is the instantaneous cost of elapsed time, while the scalar product with the controlled vector field measures the effect of the learning effort on the future value of the state. The dependence on $m(t)$ reflects the fact that the productivity of learning depends on the aggregate educational environment.
Formally, the value function therefore satisfies the Hamilton--Jacobi equation
\begin{equation}\label{eq:HJB}
\partial_tV(t,x) + H\bigl(t, x, m(t), \nabla_x V(t,x)\bigr)
= 0, \qquad x\in\mathbb R_+^2\setminus\Gamma,
\end{equation}
with the target condition \(V(t,x)=0\), for \(x\in\Gamma\). The boundary $b=0$ corresponds to the state constraint generated by the budget dynamics, and consequently the appropriate analytical framework is that of state-constrained viscosity solutions. In general, as discussed earlier, the value function is not expected to be differentiable, and equation~\eqref{eq:HJB} should therefore be interpreted in the viscosity sense. More precisely, the value function is naturally characterized in the framework of state-constrained viscosity solutions of the Hamilton--Jacobi equation, satisfying the target condition on $\Gamma$. The viscosity formulation is particularly important because the value function may fail to be smooth near the target set or near the boundary generated by the budget constraint.
The Hamiltonian which is introduced in Definition~\ref{def: Hamiltonian}, provides a direct characterization of the optimal learning effort. To obtain an explicit feedback law, we impose additional convexity and differentiability assumptions on the effort cost, to establish the optimal feedback uniquely.

\begin{itemize}
    \item[\rm (H5)]\label{ass:feedback} The effort cost function
    $c\colon[0,1]\longrightarrow\mathbb R_+$
    belongs to $C^2([0,1])$ and satisfies
    \(c''(u)>0\), for all \(u\in(0,1)\).
    In particular, $c$ is strictly convex on $[0,1]$.
\end{itemize}
Under this assumption, the minimization problem defining the Hamiltonian has a unique solution.

\begin{theorem}[Optimal feedback characterization]
\label{thm:feedback}
Assume that {\rm (H1)--(H5)} hold. Let $(t,x,p)\in[0,+\infty)\times\mathbb R_+^2\times\mathbb R^2$ and suppose that $p_b<0$. Then the minimization problem
\[
H(t, x, m(t), p) = \inf_{u\in[0,1]}\left\{1 + K(m(t),z)u\,p_z + \bigl(w(z)-c(u)\bigr)p_b \right\}
\]
admits a unique minimizer $u^{*,m}(t,x,p)\in[0,1]$.
If the minimizer is interior, i.e.\@ \(0 < u^{*,m}(t,x,p)\allowbreak <1\), then it satisfies the first-order optimality condition \(K(m(t),z)p_z - c'(u^{*,m})p_b = 0\). Equivalently,
\[
u^{*,m} = (c')^{-1} \left(\frac{K(m(t),z)p_z}{p_b} \right).
\]
More generally, the optimal feedback is obtained by projecting the unconstrained minimizer onto the admissible interval,
\[
u^{*,m}(t, x, p) = \Pi_{[0,1]}\left[(c')^{-1}\left(
\frac{K(m(t),z)p_z}{p_b}\right)\right],
\]
whenever the inverse is defined on the relevant range.
\end{theorem}

\begin{proof}
For a given measurable distribution \(m(\cdot)\) and a fixed $(t,x,p)$, we define \(\Phi(u) = 1 + K(m(t),z)u\,p_z + \bigl(w(z)-c(u)\bigr)p_b\), for \( u\in[0,1]\). The terms independent of $u$ do not affect the minimization. Differentiating gives
\(\Phi'(u) = K(m(t),z)\allowbreak p_z - c'(u)p_b\) and \(\Phi''(u) = -c''(u)p_b\). By assumption \rm (H5), one observes that $c''(u)>0$, and since $p_b<0$, we have $\Phi''(u)>0$ for every $u\in(0,1)$. Hence $\Phi$ is strictly convex on the compact interval $[0,1]$ and therefore admits a unique global minimizer. If the minimizer is interior, the first-order necessary condition $\Phi'(u^*)=0$ yields 
\[
c'(u^*) = \frac{K(m(t),z)p_z}{p_b}.
\]
Because $c''>0$, the derivative $c'$ is strictly increasing and therefore invertible on its image, which yields us to obtain
\[
u^* = (c')^{-1} \left(\frac{K(m(t),z)p_z}{p_b}
\right).
\]
If the unconstrained minimizer lies outside $[0,1]$, strict convexity implies that the constrained minimum is attained at the nearest endpoint, which gives the projected feedback representation, as stated in the Theorem.
\end{proof}

The feedback law, stated in the previous theorem, naturally separates the optimization problem into two regimes, as follows. 
\vspace{-5pt}
\paragraph*{Interior regime.}
When the unconstrained minimizer belongs to the interval $(0,1)$, the optimal effort balances the marginal value of faster human-capital accumulation against the marginal value of financial resources. When \(V\) is differentiable at the point under consideration, the first-order condition \(K(m(t),z)\,\partial_z V = c'(u^*)\,\partial_b V\) shows that the marginal benefit of increasing the learning effort is exactly equal to its marginal budgetary cost. In the general viscosity setting, the corresponding characterization is expressed in terms of the gradient \(p\) appearing in the Hamiltonian minimization.
\vspace{-10pt}
\paragraph*{Boundary regime.}
When the unconstrained minimizer falls outside the admissible interval, the optimal control is constrained by the physical limits of educational effort. The agent either chooses no learning effort
($u^*=0$) or maximal learning effort ($u^*=1$), which is also known as a \emph{bang--bang} situation. Such boundary regimes arise when the shadow value of human-capital accumulation is sufficiently small or sufficiently large relative to the shadow value of financial resources. For more discussion on the bang--bang type controls, we refer the reader to work~\cite[Section~5]{Arjmand_2026}, in which these types are studied for a biological model. 
\vspace{10pt}

The feedback characterization provides a direct economic interpretation of the optimal policy. The derivative $\partial_zV$ measures the value of increasing human capital, while $-\partial_bV$ measures the value of relaxing the budget constraint. The learning technology $K(m(t),z)$ amplifies the return to educational effort through the surrounding population environment. Consequently, the optimal effort is determined by the interaction between individual incentives and the aggregate educational conditions represented by the population distribution $m(t)$. The optimal control is thus obtained by evaluating the feedback law at the gradient of the value function in the viscosity sense, yielding
\(u^{*,m}(t,x) = u^{*,m}\bigl(t,x,\nabla_xV^m(t,x)\bigr)\). Since \(V^m\) is understood in the viscosity sense and does not need to be differentiable everywhere, the feedback \(u^{*,m}\) can be selected measurably from the set of Hamiltonian minimizers. This representation will play a central role in the construction of the velocity field governing the evolution of the population distribution in the MFG equilibrium in Section~\ref{sec:MFG-system}. However, let us continue the current section with some further properties of the value function.


The value function possesses several qualitative properties that reflect the economic structure of the model. In particular, larger initial human capital or financial resources should not increase the time required to reach the qualification threshold. We first establish the continuous dependence of trajectories on their initial conditions and derive the continuity properties of the value function. We then prove a comparison principle for the controlled dynamics, deduce the monotonicity of the value function with respect to the initial state, and finally study comparative statics with respect to the structural primitives of the model. We begin the passage with the stability of trajectories with respect to their initial conditions.

\begin{proposition}[Continuous dependence of trajectories]
\label{prop:continuous_dependence}
Assume that {\rm (H1)--(H4)} hold. Fix $t_0\ge0$ and let
$x_0^1,x_0^2\in\mathbb R_+^2$. Let
$u\in\mathcal U(t_0,x_0^1)\cap\mathcal U(t_0,x_0^2)$,
and denote by
$x_i(t)=x_{t_0,x_0^i}^{u}(t)$, for $i=1,2$,
the corresponding trajectories. Then there exists a constant $C>0$, depending only on the Lipschitz constants of the controlled vector field, such that
\[
|x_1(t)-x_2(t)| \le C e^{C(t-t_0)} |x_0^1-x_0^2|,
\quad \forall\, t\ge t_0.
\]
\end{proposition}

\begin{proof}
The controlled vector field associated with \eqref{eq:control-system} is measurable in time and Lipschitz continuous with respect to the state variable under {\rm (H1)--(H4)}. Hence there exists $C>0$ such that \(|f(t,x_1,u)-f(t,x_2,u)|
\le C|x_1-x_2|\). Using the integral formulation of the trajectories, we have
\[
|x_1(t)-x_2(t)| \le |x_0^1-x_0^2| + C\int_{t_0}^t
|x_1(s)-x_2(s)|\,\diff s.
\]
Gronwall's inequality then yields \(|x_1(t)-x_2(t)| \leq e^{C(t-t_0)} |x_0^1-x_0^2|\), for which renaming the constant gives the claimed estimate.
\end{proof}

The previous estimate immediately implies the continuity of the value function with respect to the initial state.

\begin{proposition}[Continuity of the value function]
\label{prop:V-continuity}
Assume that \emph{(H1)--(H4)} hold. Fix \(t_0\ge0\). Then the value function \(x\longmapsto V(t_0,x)\) is lower semicontinuous on \(\mathbb R_+^2\). Moreover, let \(x\in\operatorname{int}\mathcal R(t_0)\), and assume that there exists an optimal control \(u^{*,m}\in\mathcal U(t_0,x)\) whose admissibility and target-reaching property are stable under sufficiently small perturbations of the initial state. Then \(V(t_0,\cdot)\) is continuous at \(x\).
\end{proposition}

\begin{proof}
We first prove the lower semicontinuity of \(V(t_0,\cdot)\). Let \(x_n\to x\) in \(\mathbb R_+^2\), and  set \(\ell:=\liminf_{n\to\infty}V(t_0,x_n)\). If \(\ell=+\infty\), there is nothing to prove. Otherwise, after extracting a subsequence, we may assume that \(V(t_0,x_n)\longrightarrow \ell<+\infty\). For each \(n\), let \(u_n\in\mathcal U(t_0,x_n)\) be an optimal control, whose existence follows from Theorem~\ref{thm:existence}, and denote the corresponding trajectory by \(x_n(t)=x_{t_0,x_n}^{u_n}(t)\). Since the hitting times \(V(t_0,x_n)\) remain bounded, all trajectories reach the target before some common finite time \(T>t_0\). Extending the controls by zero after their respective hitting times, we may regard them as bounded elements of \(L^\infty(t_0,T)\). By the Banach--Alaoglu theorem, after extraction of a subsequence, \(u_n\rightharpoonup \tilde u\) in \(L^\infty(t_0,T)\). The uniform growth estimate in {\rm(H4)} and the corresponding equicontinuity estimate yield, by the Arzel\`a--Ascoli theorem, uniform convergence \(x_n\longrightarrow x^\star\) on \([t_0,T]\). The same compactness argument as in the proof of Theorem~\ref{thm:existence} shows that \(x^\star\) is the trajectory associated with \(\tilde u\). Moreover, since the state constraint is preserved under uniform convergence, \(\tilde u\) is admissible from \(x\). Since \(x_n\bigl(t_0+V(t_0,x_n)\bigr)\in\Gamma\) and \(\Gamma\) is closed, the uniform convergence of the trajectories and the convergence of the hitting times imply \(x^\star(t_0+\ell)\in\Gamma\). Consequently,
\(V(t_0,x)\le\tau(t_0,x,\tilde u)\le\ell =\liminf_{n\to\infty}V(t_0,x_n)\), and thus \(V(t_0,\cdot)\) is lower semicontinuous.

Now let \(x\in\operatorname{int}\mathcal R(t_0)\), and assume that there exists an optimal control \(u^{*,m}\in\mathcal U(t_0,x)\) such that, for all initial states \(y\) sufficiently close to \(x\), the same control \(u^{*,m}\) is admissible from \(y\), and its hitting time depends continuously on the initial state. By continuous dependence of the trajectories on the initial state,
\(\tau(t_0,y,u^{*,m})\longrightarrow \tau(t_0,x,u^{*,m})=V(t_0,x)\) as \(y\longrightarrow x\). Since \(V(t_0,y)\le\tau(t_0,y,u^{*,m})\), we obtain
\(\limsup_{y\to x}V(t_0,y)\le V(t_0,x)\). Together with the lower semicontinuity already proved, this yields
\(V(t_0,y)\longrightarrow V(t_0,x)\) as \(y\longrightarrow x\). Hence \(V(t_0,\cdot)\) is continuous at \(x\).
\end{proof}

We now establish the monotonicity of the value function with respect to the initial state.

\begin{theorem}[Comparison principle]
\label{thm:comparison}
Assume that {\rm (H1)--(H4)} hold and that, for every \(m\in\mathcal P(\mathbb R_+^2)\), the functions \(z\longmapsto K(m,z)\) and \(z\longmapsto w(z)\) are nondecreasing. Let \(t_0\ge0\), and \(x_0^1=(z_0^1,b_0^1)\), \(x_0^2=(z_0^2,b_0^2)\) satisfy \(z_0^1\le z_0^2\), \(b_0^1\le b_0^2\). Then \(V(t_0,x_0^2)\le V(t_0,x_0^1)\). In particular, for every fixed initial time \(t_0\), the value function is nonincreasing with respect to both components of the initial state.
\end{theorem}

\begin{proof}
Fix \(t_0\ge0\) and let \(u\in\mathcal U(t_0,x_0^1)\) be an admissible control for the problem starting from \(x_0^1\) at time \(t_0\). We apply the same control to the system starting from \(x_0^2\). Denote the corresponding trajectories by \(x_i(t)=\bigl(z_i(t),b_i(t)\bigr)\) for which the control system~\eqref{eq:control-system} is satisfied, where \(i=1,2\) and \(t\ge0\). We first compare the human-capital components. Since
\(u(t)\in[0,1]\) and, for every fixed measure \(m\), the function \(z\mapsto K(m,z)\) is nondecreasing, the scalar vector field \(z\longmapsto K\bigl(m(t_0+t),z\bigr)u(t)\) is nondecreasing for almost every \(t\ge0\). Since \(z_0^1\le z_0^2,\) the comparison principle for scalar differential equations yields \(z_1(t)\le z_2(t)\), for all \(t\ge0\). Because \(w\) is nondecreasing, we consequently have
\(w\bigl(z_1(t)\bigr) \le w\bigl(z_2(t)\bigr)\), for all \(t\ge0\). The two trajectories are driven by the same control, and therefore we have
\[
\dot b_1(t) =  w\bigl(z_1(t)\bigr)-c\bigl(u(t)\bigr) \le
w\bigl(z_2(t)\bigr)-c\bigl(u(t)\bigr) = \dot b_2(t).
\]
Together with \(b_0^1\le b_0^2\) implies
\(b_1(t)\le b_2(t)\), for all \(t\ge0\). Thus, under the same control, the trajectory starting from the larger initial state dominates the trajectory starting from the smaller initial state componentwise, which means that \(z_1(t)\le z_2(t)\) and \(b_1(t)\le b_2(t)\) for all \(t\ge0\).

This comparison also implies the inclusion of admissible control sets. Indeed, since \(u\in\mathcal U(t_0,x_0^1)\), the corresponding trajectory satisfies \(b_1(t)\ge0\), for all \(t\ge0\). Since \(b_2(t)\ge b_1(t)\), we obtain \(b_2(t)\ge0\), for all \(t\ge0\). Hence the same control is admissible from the second initial state  \(u\in\mathcal U(t_0,x_0^2)\). Therefore,
\(\mathcal U(t_0,x_0^1) \subseteq \mathcal U(t_0,x_0^2)\). Moreover, since \(z_1(t)\le z_2(t)\), for all \(t\ge0\), the trajectory starting from \(x_0^2\) reaches the target set no later than the trajectory starting from \(x_0^1\). More precisely, \(\tau(t_0,x_0^2,u) \le \tau(t_0,x_0^1,u)\) for every \(u\in\mathcal U(t_0,x_0^1)\). Taking the infimum over the smaller admissible set gives
\[
V(t_0,x_0^2) = \inf_{u\in\mathcal U(t_0,x_0^2)}
\tau(t_0,x_0^2,u) \le \inf_{u\in\mathcal U(t_0,x_0^1)}
\tau(t_0,x_0^2,u) \le \inf_{u\in\mathcal U(t_0,x_0^1)} \tau(t_0,x_0^1,u) = V(t_0,x_0^1),
\]
which proves the result.
\end{proof}

The comparison principle has a direct economic interpretation. Starting from a higher level of human capital or from a larger initial budget cannot make the qualification objective more difficult to attain. A higher initial skill level places the agent closer to the target and, under the monotonicity of the income function, generates at least as much income throughout the evolution. A larger initial budget, on the other hand, relaxes the state constraint and enlarges the set of feasible learning policies. These conclusions can be stated explicitly as follows.

\begin{corollary}[Monotonicity of the value function]
\label{cor:monotonicity}
Under the assumptions of Theorem~\ref{thm:comparison}, for every fixed \(t_0\ge0\), the following statements are satisfied for the value function.
\begin{enumerate}
    \item[{\rm (i)}] For every fixed \(b\ge0\), the mapping
    \(z\longmapsto V(t_0,z,b)\) is nonincreasing.
    
    \item[{\rm (ii)}] For every fixed \(z\ge0\), the mapping
    \(b\longmapsto V(t_0,z,b)\) is nonincreasing.
    
    \item[{\rm (iii)}] More generally, if \((z_1,b_1)\le(z_2,b_2)\) componentwise, then
    \(V(t_0,z_2,b_2) \le V(t_0,z_1,b_1)\).
\end{enumerate}
\end{corollary}

\begin{proof}
Assertion {\rm (iii)} is precisely the conclusion of
Theorem~\ref{thm:comparison}. To obtain {\rm (i)}, fix \(b\ge0\) and let \(z_1\le z_2\). Applying the comparison theorem to \(x_0^1=(z_1,b)\) and \(x_0^2=(z_2,b)\) gives \(V(t_0,z_2,b) \le V(t_0,z_1,b)\). Hence \(z\mapsto V(t_0,z,b)\) is nonincreasing. Similarly for assertion {\rm (ii)}, fixing \(z\ge0\) and taking \(b_1\le b_2\), the comparison theorem applied to \(x_0^1=(z,b_1)\) and \(x_0^2=(z,b_2)\) yields \(V(t_0,z,b_2) \le V(t_0,z,b_1)\). Therefore \(b\mapsto V(t_0,z,b)\) is nonincreasing.
\end{proof}

The preceding results describe the dependence of the optimal completion time on the initial state of an individual. We now turn to a different type of comparative statics question, regarding the change in the value function when the structural environment governing human-capital accumulation is modified.
The dynamics are determined by three fundamental primitives such as the learning technology \(K\), the income function \(w\), and the effort cost \(c\). The learning technology determines the productivity of educational effort, while the income function determines the financial resources generated by the current level of human capital, and the effort cost determines the budgetary burden associated with a given learning intensity. Modifying any of these primitives changes the set of feasible trajectories and, potentially, the set of optimal strategies. The following result compares the corresponding minimal-time problems. In each case, the initial time \(t_0\), the initial state \(x_0\), and the measurable population distribution \(m(\cdot)\) are kept fixed. For \(i=1,2\), we denote by \(V_i(t_0,x_0)\) the value functions associated with the respective primitives.

\begin{theorem}[Comparative statics]
\label{thm:comparative}
Assume that the relevant pairs of primitives satisfy
{\rm (H1)--(H4)}, together with the monotonicity assumptions required in Theorem~\ref{thm:comparison}. Then the value function is monotone with respect to the structural primitives of the model in the following sense.

\begin{enumerate}
    \item[{\rm (i)}]
    Let \(K_1\) and \(K_2\) be two learning technologies satisfying \(K_1(m,z)\le K_2(m,z)\) for all \((m,z)\in\mathcal P(\mathbb R_+^2)\times\mathbb R_+\). Let \(V_1\) and \(V_2\) denote the corresponding value functions. Then, for every \(t_0\ge0\) and every \(x_0\in \mathbb R_+^2\), we have \(V_2(t_0,x_0) \le V_1(t_0,x_0)\).
    Thus, an improvement in the learning technology cannot increase the minimal time required to reach the qualification threshold.
    
    \item[{\rm (ii)}]
    Let \(w_1,w_2\colon \mathbb R_+\longrightarrow\mathbb R_+\) satisfy
    \(w_1(z)\le w_2(z)\), for all \(z\ge0\). Let \(V_1\) and \(V_2\) denote the corresponding value functions. Then, for every \(t_0\ge0\) and every \(x_0\in \mathbb R_+^2\), we have \(V_2(t_0,x_0) \le V_1(t_0,x_0)\). Hence, increasing the income generated by human capital cannot increase the minimal time required to attain the target.
    
    \item[{\rm (iii)}]
    Let \(c_1,c_2\colon[0,1]\longrightarrow\mathbb R_+\) satisfy \(c_1(u)\le c_2(u)\), for all \(u\in[0,1]\). Let \(V_1\) and \(V_2\) be the corresponding value functions. Then, for every \(t_0\ge0\) and every \(x_0\in \mathbb R_+^2\), we have \(V_1(t_0,x_0) \le V_2(t_0,x_0)\). Therefore, increasing the cost of learning effort cannot decrease the minimal time required to reach the qualification threshold.
\end{enumerate}
\end{theorem}

\begin{proof}
As declared before, we first fix an initial time \(t_0\ge0\) and an initial state \(x_0=(z_0,b_0)\in \mathbb R_+^2\), and then we compare the corresponding controlled systems by applying the same control to both systems.

\paragraph*{(i) Improvement in the learning technology.}

Let \(u\) be admissible for the system associated with \(K_1\), and denote the corresponding trajectory by
\(x_1(t)=\bigl(z_1(t),b_1(t)\bigr)\). Let \(x_2(t)=(z_2(t),b_2(t))\) denote the trajectory generated by the same control under \(K_2\), for which the human-capital dynamics are \(\dot z_i(t) = K_i\bigl(m(t_0+t),z_i(t)\bigr)u(t)\), for \(i=1,2\). Since \(K_1(m,z)\le K_2(m,z)\) and \(z\mapsto K_1(m,z)\) is nondecreasing, the scalar comparison principle yields \(z_1(t)\le z_2(t)\), for all \(t\ge0\). Since \(w\) is nondecreasing, we consequently have \(w\bigl(z_1(t)\bigr) \le w\bigl(z_2(t)\bigr)\). Therefore,
\[
\dot b_1(t) = w\bigl(z_1(t)\bigr)-c(u(t)) \le w\bigl(z_2(t)\bigr)-c(u(t)) = \dot b_2(t).
\]
Because the two trajectories start from the same initial budget, \(b_1(0)=b_2(0)=b_0\), we obtain \(b_1(t)\le b_2(t)\), for all \(t\ge0\). Thus, every control admissible for the system with \(K_1\) is also admissible for the system with \(K_2\). Moreover, the improved learning technology produces a human-capital trajectory that reaches the target no later. Hence \(\tau_2(t_0,x_0,u) \le \tau_1(t_0,x_0,u)\). Taking the infimum over all controls admissible for the first system gives
\[
V_2(t_0,x_0) = \inf_{u\in\mathcal U_2(t_0,x_0)}
\tau_2(t_0,x_0,u) \le \inf_{u\in\mathcal U_1(t_0,x_0)}
\tau_2(t_0,x_0,u) \le \inf_{u\in\mathcal U_1(t_0,x_0)}
\tau_1(t_0,x_0,u) = V_1(t_0,x_0).
\]

\paragraph*{(ii) Improvement in the income function.}

Let \(u\) be admissible for the system associated with \(w_1\). Since the human-capital dynamics do not depend directly on the income function, the two trajectories generated by the same control satisfy \(z_1(t)=z_2(t)\), for all \(t\ge0\). Since the budget dynamics are \(\dot b_i(t) = w_i\bigl(z_i(t)\bigr)-c(u(t))\), for \(i=1,2\), and \(w_1(z)\le w_2(z)\), for all \(z\ge0\), we obtain
\(\dot b_1(t)\le\dot b_2(t)\). Together with \(b_1(0)=b_2(0)=b_0\),  this implies \(b_1(t)\le b_2(t)\), for all \(t\ge0\).  Therefore, \(\mathcal U_1(t_0,x_0) \subseteq \mathcal U_2(t_0,x_0)\). For every control admissible under \(w_1\), the human-capital trajectories are identical under the two income functions. Hence \(\tau_2(t_0,x_0,u) = \tau_1(t_0,x_0,u)\). Taking the infimum over the respective admissible sets yields
\[
V_2(t_0,x_0) = \inf_{u\in\mathcal U_2(t_0,x_0)} \tau_2(t_0,x_0,u) \le \inf_{u\in\mathcal U_1(t_0,x_0)}
\tau_2(t_0,x_0,u) = \inf_{u\in\mathcal U_1(t_0,x_0)}
\tau_1(t_0,x_0,u) = V_1(t_0,x_0).
\]

\paragraph*{(iii) Increase in the cost of learning effort.}

Let \(u\) be admissible for the system associated with the higher cost function \(c_2\). Since the human-capital dynamics do not depend directly on the effort cost, the corresponding human-capital trajectories coincide, and therefore we have \(z_1(t)=z_2(t)\), for all \(t\ge0\). For \(i=1,2\), the budget dynamics satisfy \(\dot b_i(t) = w\bigl(z_i(t)\bigr)-c_i(u(t))\), for which, since \(c_1(u)\le c_2(u)\) for all \(u\in[0,1]\), we have
\(\dot b_1(t)\ge\dot b_2(t)\). Because \(b_1(0)=b_2(0)=b_0\), it follows that \(b_1(t)\ge b_2(t)\), for all \(t\ge0\). Consequently,
\(\mathcal U_2(t_0,x_0) \subseteq \mathcal U_1(t_0,x_0) \). For every control admissible under the higher cost \(c_2\), the human-capital trajectories are identical under the two cost functions. Therefore, \(\tau_1(t_0,x_0,u) = \tau_2(t_0,x_0,u)\). Taking the infimum over the admissible controls for the second problem gives
\[
V_1(t_0,x_0) = \inf_{u\in\mathcal U_1(t_0,x_0)} \tau_1(t_0,x_0,u) \le \inf_{u\in\mathcal U_2(t_0,x_0)}
\tau_1(t_0,x_0,u) = \inf_{u\in\mathcal U_2(t_0,x_0)}
\tau_2(t_0,x_0,u) = V_2(t_0,x_0).
\]
Hence, by the three assertions, we observe that the improvements in learning technology or in the income generated by human capital weakly reduce the minimal time required to reach the qualification threshold, whereas an increase in the cost of learning effort weakly increases it. In other words, larger initial human capital, larger financial resources, more productive learning technologies, and higher income opportunities all reduce the time required to attain the qualification threshold, whereas higher educational costs increase it.
\end{proof}

The analysis of the individual minimum-time problem developed in this section provides the main ingredients needed to construct the coupled MFG. In particular, the well-posedness of the controlled dynamics, the existence of optimal controls, the dynamic programming characterization of the value function, and the associated optimal feedback will play a central role in the analysis of the population dynamics. These properties allow us to define the best-response operator that maps a given population flow into the population trajectory induced by the corresponding optimal individual decisions. We now use these
individual-level results to establish the existence of an MFG equilibrium on an arbitrary finite time interval.

\section{Existence of an MFG equilibrium on a finite time interval}\label{sec:MFG-system}

In Section~\ref{sec: Human-capital model}, we analyzed the individual minimal-time optimal control problem for a fixed population flow and characterized the corresponding value function through the DPP and the Hamilton--Jacobi equation. In this section, we endogenize the population distribution and formulate the MFG equilibrium. Unlike in the individual problem, where the population flow is treated as fixed, here the population distribution must be generated consistently by the optimal behavior of the agents. In other words, the distribution \(m(\cdot)\) is no longer treated as an exogenous input, but must be generated by the optimal behavior of the population itself. Each agent chooses an optimal learning policy given the current population distribution, while the population evolves under the velocity field induced by these optimal decisions. An equilibrium is therefore characterized by a consistency condition between individual optimization and aggregate population dynamics. In other words, the equilibrium problem consists in finding a population flow \(m\) which is mutually consistent with the value function \(V^{m}\) and an optimal feedback \(u^{*,m}\), defined in Section~\ref{sec: Human-capital model}.
The individual optimization problem determines the optimal learning effort for each agent, while the aggregate dynamics describe the evolution of the distribution of agents in the state space \(\mathbb R_+^2\). Since the learning technology depends on the population distribution, the individual and aggregate problems are coupled through the distribution \(m(t)\).

The equilibrium condition expresses the fact that the population distribution used by the agents when solving their individual optimization problem coincides with the distribution generated by their optimal trajectories. Therefore, this consistency naturally encourages us to combine the Hamilton--Jacobi equation with the continuity equation in order to yield the coupled \emph{MFG system}
\begin{equation}\label{eq:MFG-system}
\left\{
\begin{aligned}
&\partial_t V^m (t,x) + H\!\left(t,x,m(t),\nabla V^{m}(t,x)\right)=0,
&& (t,x)\in(0, T)\times (\mathbb R_+^2\setminus\Gamma),
\\[1ex]
&\partial_t m(t) +\diverg_{x} \left(F^m(t,x)\,m(t)\right)=0,
&& (t,x)\in(0,T)\times\mathbb R_+^2,
\\[1ex]
&V^{m}(t,x)=0,
&& x\in\Gamma,
\\[1ex]
&m(0)=m_0,
&& \text{in }\mathbb R_+^2,
\end{aligned}
\right.
\end{equation}
where \( F^m(t,x) = f\!\left(t,x,m(t), u^{*,m}(t,x) \right)\). Although the individual problem is posed on the infinite time horizon, the equilibrium construction below is carried out on an arbitrary finite interval $[0,T]$. The system \eqref{eq:MFG-system} couples the individual optimization problem with the aggregate population dynamics. The Hamilton--Jacobi equation determines the optimal learning policy of a representative agent for a given population distribution, while the continuity equation describes the evolution of the population under the optimal feedback generated by the value function. The coupling is bidirectional, based on the fact that the population distribution affects the learning technology through the coefficient \(K(m(t),z)\), and the value function determines the velocity field that transports the population. An MFG equilibrium is therefore a fixed point of this interaction between individual incentives and aggregate dynamics (see, Definition~\ref{def: MFG equilibrium}).

The main difficulty is that the population distribution is simultaneously an input and an output of the individual optimization problem. Given a population flow \(m\), the representative agent solves the minimal-time optimal control problem and obtains a value function \(V^m\) and an associated optimal feedback \(u^{*,m}\). This feedback determines the velocity field of the population and therefore generates a new population flow. An equilibrium is obtained when the population flow used by agents to optimize coincides with the population flow generated by their optimal behavior.
Since the model is formulated on an infinite time horizon, we first establish the existence of equilibria on arbitrary finite intervals \([0,T]\). The finite-horizon setting provides the compactness required for the fixed-point argument and allows the population dynamics to be analyzed in a complete metric space. The finite-horizon equilibria provide the natural starting point for a possible infinite-horizon analysis. Establishing an infinite-horizon equilibrium would require additional estimates and a compactness argument that are beyond the scope of the present work. In other words, the individual optimization problem, established in Section~\ref{sec: Human-capital model}, is formulated on an infinite time horizon, since the objective is the first hitting time of the qualification set and no terminal time is imposed a priori. Nevertheless, the construction of an MFG equilibrium is first carried out on an arbitrary finite time interval \([0,T]\). For each \(T > 0\), the population dynamics are considered on \([0,T]\), and the corresponding fixed-point problem is solved in the space of continuous population flows on this interval. The parameter \(T\) therefore represents the time interval on which the population equilibrium is constructed, rather than a terminal time in the individual minimum-time problem. The resulting finite-horizon equilibria can subsequently be investigated as \(T\longrightarrow \infty\), with the aim of obtaining an infinite-horizon equilibrium under suitable uniform estimates.

Throughout this section, we fix a finite time horizon \(T>0\),
and introduce the subset of probability measures with finite first moment,
\[
\mathcal P_1(\mathbb R_+^2) = \left\{\mu\in\mathcal P(\mathbb R_+^2) \suchthat \int_{\mathbb R_+^2}|x|\,\mu(\diff x)<+\infty
\right\},
\]
where \(|x|=(z^2+b^2)^{1/2}\) denotes the Euclidean norm. The natural metric on \(\mathcal P_1(\mathbb R_+^2)\) is the Wasserstein distance of order one, as follows.
\[
W_1(\mu,\nu) = \inf_{\pi\in\Pi(\mu,\nu)} \int_{\mathbb R_+^2\times\mathbb R_+^2} |x-y|\, \pi(\diff x,\diff y),
\]
where \(\Pi(\mu,\nu)\) denotes the set of couplings of \(\mu\) and \(\nu\). The metric space \((\mathcal P_1(\mathbb R_+^2),W_1)\) is complete and separable.
We denote by \(\mathcal C_T = C\bigl([0,T]; \mathcal P_1(\mathbb R_+^2)\bigr)\) the space of continuous probability measures endowed with the metric
\[
d_T(m^1,m^2) = \sup_{t\in[0,T]} W_1\bigl(m^1(t),m^2(t)\bigr).
\]
Since \((\mathcal P_1(\mathbb R_+^2),W_1)\) is complete, \((\mathcal C_T,d_T)\) is also a complete metric space.
We fix an initial population distribution \(m_0\in\mathcal P_1(\mathbb R_+^2)\). For every $m\in\mathcal C_T$, we consider the individual minimum-time problem with the measurable population $m(\cdot)$ on $[0,T]$, for which the individual objective remains the first hitting time of $\Gamma$, rather than a terminal-time criterion, in which we still denote the associated value function by $V^m$ for simplicity.
Notice that we assume the structural hypotheses \rm {(H1)--(H4)}, introduced in Section~2, are satisfied, and in addition, we impose the following assumptions (namely, (E1) and (E2)) on the population interaction and on the regularity of the extended vector field.

\begin{enumerate}
    \item[(E1)] \textbf{Lipschitz continuity with respect to the population measure.}
    There exists a constant $L_m>0$ such that
    \[
    |K(\mu,z)-K(\nu,z)| \le L_m\,W_1(\mu,\nu),
    \]
    for all $\mu,\nu\in\mathcal P_1(\mathbb R_+^2)$
    and all $z\in\mathbb R_+$.
\end{enumerate}
Assumption (E1) ensures that the population interaction enters the dynamics continuously with respect to the Wasserstein topology. 
The value function of the individual optimization problem is defined through the first hitting time of the qualification set \(\Gamma\). Consequently, the optimal feedback \(u^{*,m}\) is only relevant on the complement \(\mathbb R_+^2\setminus\Gamma\), where the optimization problem is active. In order to define the population dynamics globally on the entire state space, it is necessary to specify the behavior of agents after they have reached the qualification threshold. We adopt the conservative population formulation introduced in the previous section. Qualified agents remain part of the population distribution, and their subsequent evolution is fixed by a measurable extension of the optimal feedback.
Let \(m\in\mathcal C_T\) be a measurable population flow. Assume that \(u^{*,m}\colon [0,T]\times (\mathbb R_+^2\setminus\Gamma) \longrightarrow [0,1]\) is the optimal feedback associated with the value function \(V^m\). We fix a measurable function
\(\bar u \colon [0,T]\times\Gamma \longrightarrow [0,1]\), called the \emph{post-qualification feedback}, and define the extended feedback \(\widehat u^{\,m}\colon  [0,T]\times\mathbb R_+^2 \longrightarrow [0,1]\) by
\begin{equation*}
\label{eq:extended-feedback}
\widehat u^{\,m}(t,x) =
\begin{cases}
u^{*,m}(t,x),& x\in \mathbb R_+^2\setminus\Gamma,
\\[0.8ex]
\bar u(t,x),& x\in\Gamma.
\end{cases}
\end{equation*}

Since \(\Gamma\) is a Borel subset of \(\mathbb R_+^2\), and both \(u^{*,m}\) and \(\bar u\) are measurable, the function \(\widehat u^{\,m}\) is measurable on \([0,T]\times\mathbb R_+^2\). The choice of \(\bar u\) does not affect the individual optimization problem, because the latter terminates at the first hitting time of \(\Gamma\). Its only purpose is to define the subsequent evolution of qualified agents in the aggregate population dynamics. A natural choice is \(\bar u\equiv0\), which corresponds to agents ceasing educational effort after qualification, but the analysis below only requires measurability and boundedness of the extension.  The extended feedback induces the following vector field \(\widehat F^m(t,x) = f\!\left(t,x,m(t),\widehat u^{\,m}(t,x) \right)\), for which we make the following hypothesis.  
\begin{enumerate}
    \item[(E2)] \textbf{Regularity of the extended closed-loop vector field.}
    For every measurable population flow $m\in\mathcal C_T$, the extended closed-loop vector field
    \(\widehat F^m(t,x)\) is measurable in $t$ and globally Lipschitz continuous with respect to $x$, uniformly with respect to $t\in[0,T]$ and $m\in\mathcal C_T$. More precisely, there exists a constant $L_F>0$ such that
    \[
    |\widehat F^m(t,x)-\widehat F^m(t,y)| \le L_F|x-y|,
    \]
    for all $x,y\in\mathbb R_+^2$, all $t\in[0,T]$, and all $m\in\mathcal C_T$.
    This assumption guarantees the well-posedness of the characteristic flow associated with the continuity equation. More precisely, since \(0\le\widehat u^{\,m}\le1\), assumption (H4) implies the linear growth estimate \(|\widehat F^m(t,x)| \le C \left(1+|x|\right)\), for some constant \(C>0\) independent of \(m\) and \(t\in[0,T]\). 
\end{enumerate}

\begin{proposition}[Well-posedness of the extended closed-loop dynamics] \label{prop:closed-loop-wellposed}
Assume that {\rm (H1)--(H4)} and {\rm (E1)--(E2)} hold. Let \(m\in\mathcal C_T\) and \(x_0\in\mathbb R_+^2\). Then the nonautonomous differential equation
\begin{equation}
\label{eq:closed-loop-system}
\left\{
\begin{aligned}
\dot x(t) &= \widehat F^m(t,x(t)), \qquad t\in[0,T],\\
x(0) &= x_0,
\end{aligned}
\right.
\end{equation}
admits a unique absolutely continuous solution \(x_{0,x_0}(t) \in AC([0,T];\mathbb R_+^2)\). Furthermore, there exists a constant \(C_T>0\), depending only on \(T\) and the structural constants, such that
\begin{equation}
\label{eq:trajectory-growth}
|x_{0,x_0}(t)| \le C_T \left(1+|x_0| \right), \qquad t\in[0,T].
\end{equation}
\end{proposition}

\begin{proof}
Let us first set  \(x(t) = x_{0,x_0}(t)\), when there is no ambiguity. By assumption {\rm (E2)}, the vector field \(\widehat F^m\) is measurable with respect to \(t\), globally Lipschitz with respect to the state variable, and satisfies the linear growth estimate established above. Therefore, the Carathéodory existence and uniqueness theorem (see, e.g.\@ \cite[Chapter 1, Theorem 1.1]{CoddingtonLevinson1955}) yields a unique absolutely continuous solution of \eqref{eq:closed-loop-system} on \([0,T]\). To prove \eqref{eq:trajectory-growth}, the linear growth estimate gives \(|\dot x(t)| \le C \left( 1+|x(t)| \right)\). Integrating this inequality and applying Gronwall's lemma,
we obtain
\[
|x(t)| \le \left(|x_0| + CT \right) e^{CT},
\]
which is equivalent to \eqref{eq:trajectory-growth} after redefining the constant \(C_T\).
\end{proof}
Proposition~\ref{prop:closed-loop-wellposed} provides a globally defined characteristic flow associated with every measurable population \(m\). This flow will be used to construct the fixed-point operator governing the evolution of the population distribution. We thus now construct the fixed-point operator associated with the MFG equilibrium problem. We consider a candidate population flow \(m\in\mathcal C_T\), for which let $V^m$ denote the value function of the individual optimal control problem, and let $\widehat u^{\,m}$ be the measurable feedback introduced  previously, while the corresponding closed-loop vector field is \(\widehat F^m\). The extended closed-loop vector field $\widehat F^m$ satisfies the Carathéodory conditions and with the global Lipschitz estimate imposed in {\rm (E2)}, which shows that for every initial condition $x_0\in\mathbb R_+^2$, the closed-loop system~\eqref{eq:closed-loop-system} admits a unique absolutely continuous solution. For every $0\le s\le t\le T$, we denote by \(\Phi^m_{s,t}(x_0)\) the corresponding flow map, that is, \(\Phi^m_{s,t}(x_0) = x(t)\) where $x(\cdot)$ is the unique solution of \eqref{eq:closed-loop-system} satisfying $x(s)=x_0$, by which the family $\bigl(\Phi^m_{s,t}\bigr)_{0\le s\le t\le T}$ satisfies the semigroup property
\[
\Phi^m_{s,t} = \Phi^m_{r,t}\circ\Phi^m_{s,r}, \qquad 0\le s\le r\le t\le T,
\]
and for \(s=t\), we denote $\Phi^m_{t,t}=\mathrm{Id}$. The evolution of the population is obtained by transporting the initial distribution $m_0$ through the flow generated by the closed-loop dynamics.

\begin{definition}[Population operator]\label{def: population operator}
For a measurable population flow $m\in\mathcal C_T$, let $\Phi^m_{s,t}$ denote the characteristic flow generated by the extended velocity field $\widehat F^m$. The \emph{population operator} is the mapping \(\mathcal T\colon\mathcal C_T\longrightarrow\mathcal C_T\), defined by \(\mathcal T(m)=\widetilde m\), where the population flow $\widetilde m$ is given by the push-forward representation
\begin{equation}\label{eq:push-forward}
\widetilde m(t)=\Phi^m_{0,t}\# m_0, \qquad t\in[0,T].
\end{equation}
Equivalently, by the definition of the push-forward, one observes that, for every Borel set $A\subset\mathbb R_+^2$, \(\widetilde m(t)(A) = m_0\!\left((\Phi^m_{0,t})^{-1}(A)\right)\).
\end{definition}

As mentioned before, the finite horizon \(T\) refers to the time interval on which the population fixed-point problem is constructed and does not represent a terminal time for the individual minimum-time problem. In particular, the value function associated with a population flow is determined by the population flow beyond time \(T\). Therefore, a population flow \(m\in\mathcal C_T\) alone does not contain sufficient information to define the infinite-horizon value function. To make the finite-horizon fixed-point construction well defined, we fix, for each \(T>0\), an extension operator \(\mathfrak E_T\colon \mathcal C_T\longrightarrow\mathcal C_\infty\) with \(\mathcal C_\infty = C\bigl([0,+\infty);\mathcal P_1(\mathbb R_+^2)\bigr)\), such that \((\mathfrak E_Tm)(t)=m(t)\), for \(t\in[0,T]\), while \(\mathfrak E_Tm\) is constant after time \(T\), namely \((\mathfrak E_Tm)(t)=m(T)\) for \(t>T\). For a measurable population flow \(m\in\mathcal C_T\), the individual minimum-time problem is then understood with the complete population flow \(\mathfrak E_Tm\). Its infinite-horizon value function and selected optimal feedback are still denoted by \(V^m\) and \(u^{*,m}\), respectively, and the finite-horizon population operator, still denoted by \(\mathcal T\colon\mathcal C_T\longrightarrow \mathcal C_T\), is defined by restricting to \([0,T]\) the population dynamics generated by the closed-loop system associated with the individual optimal feedback for \(\mathfrak E_T m\). Thus, the parameter \(T\) should be interpreted as the horizon of the population fixed-point construction, rather than as a terminal time for the individual minimum-time problem. With this convention, the fixed-point argument concerns only the population trajectory on \([0,T]\), while the individual optimization problem remains an infinite-horizon problem. In particular, whenever \(V^m\), \(u^{*,m}\), or the corresponding closed-loop dynamics are considered below, the dependence on the extension \(\mathfrak E_Tm\) is understood implicitly.

The push-forward representation of the population measures coincides with the continuity equation associated with the velocity field $\widehat F^m$. Thus, the MFG equilibrium problem can be formulated as the fixed-point equation \(m=\mathcal T(m)\). A fixed point is a measurable distribution of the population that once the individual optimization problem is defined, it generates, through the optimal feedback, exactly the same measurable distribution.

\begin{definition}[MFG equilibrium]\label{def: MFG equilibrium}
Let $T>0$. A population flow $m\in\mathcal C_T$ is called an MFG \emph{equilibrium} on \([0, T]\), if \(\mathcal T(m)=m\), where $\mathcal T$ is the population operator defined in Definition~\ref{def: population operator}. Equivalently, if $V^m$ denotes the value function associated with the population flow $m$ and $u^{*,m}$ denotes the corresponding optimal feedback, then the population flow $m$ is generated by the optimal closed-loop dynamics \(\dot x(t) = \widehat F^m(t,x(t))\) and satisfies \(m(t)=\Phi^m_{0,t}\# m_0\), for \(t\in[0,T]\). In particular, $V^m$ and $u^{*,m}$ provide the individual optimality conditions associated with the equilibrium population flow.
\end{definition}

The remainder of this part is devoted to establishing the existence of such a fixed point by means of Schauder's fixed-point theorem (see, \cite[Theorem~11.2]{Zeidler1986}). The existence of a fixed point for the operator $\mathcal T$ requires uniform estimates on the family of population trajectories generated by the extended closed-loop dynamics. In this part, we establish the fundamental growth, moment, and continuity estimates that will be used in the compactness argument. We first show that the population operator admits a measure which satisfies the continuity equation in the weak sense.

\begin{proposition}[Weak formulation of the continuity equation]\label{prop:weak-continuity}
Assume that {\rm (H1)--(H4)} and {\rm (E1)--(E2)} hold. Let $m\in\mathcal C_T$, and  $\widetilde m=\mathcal T(m)$
be defined by \eqref{eq:push-forward}. Then $\widetilde m$ is the unique weak solution of
\begin{equation}\label{eq:continuity-equation-T}
\left\{
\begin{aligned}
\partial_t\widetilde m(t) &+ \diverg_{x} \bigl(\widehat F^m(t,x)\,\widetilde m(t) \bigr)  = 0,
&&
\text{in }(0,T)\times\mathbb R_+^2,
\\
\widetilde m(0)&  = m_0.
\end{aligned}
\right.
\end{equation}
More precisely, for every test function $\varphi\in C_c^1([0,T)\times\mathbb R_+^2)$,
\begin{equation}\label{eq:weak-formulation-T}
\int_0^T \int_{\mathbb R_+^2}\left(\partial_t\varphi(t,x) +
\nabla_x\varphi(t,x) \cdot \widehat F^m(t,x) \right) \,\widetilde m(t,\diff x) \,\diff t + \int_{\mathbb R_+^2} \varphi(0,x) \,m_0(\diff x) = 0.
\end{equation}
\end{proposition}

\begin{proof}
Let $\widetilde m(t)=\Phi^m_{0,t}\#m_0$. For $\varphi\in C_c^1([0,T)\times\mathbb R_+^2)$, the chain rule along the characteristics gives
\[
\frac{\diff}{\diff t} \varphi\bigl(t,\Phi^m_{0,t}(x)\bigr) = \partial_t\varphi\bigl(t,\Phi^m_{0,t}(x)\bigr) + \nabla_x\varphi\bigl(t,\Phi^m_{0,t}(x)\bigr) \cdot \widehat F^m \bigl(t,\Phi^m_{0,t}(x)\bigr),
\]
for almost every $t\in[0,T]$. Since $\varphi(T,\cdot)=0$, integration from $0$ to $T$ and integration with respect to $m_0$ yields \eqref{eq:weak-formulation-T}. 
The push-forward representation shows that \(\widetilde m\) is a weak solution of the continuity equation associated with the characteristic flow. Under the present regularity assumptions on \(\widehat F^m\), uniqueness of weak solutions follows from the standard well-posedness theory for continuity equations with globally Lipschitz velocity fields (see, e.g.\@\cite[Chapter 8]{AmbrosioGigliSavare2008}). Hence, \(\widetilde m\) is the unique weak solution of \eqref{eq:continuity-equation-T}.
\end{proof}

\begin{proposition}[Linear growth of trajectories]
\label{prop:trajectory-growth}
Assume that {\rm (H1)--(H4)} and {\rm (E1)--(E2)} hold. There exists a constant $C_T>0$, depending only on $T$, the growth constants of $K$ and $w$, and the bound on the feedback,
such that every solution of the closed-loop system
\(\dot x(t)=\widehat F^m(t,x(t))\), with \(x(0)=x_0\in\mathbb R_+^2\), satisfies
\[
|x(t)| \le C_T\bigl(1+|x_0|\bigr), \qquad t\in[0,T].
\]
\end{proposition}

\begin{proof}
By (H4) and the bound \(0\le \widehat u^{\,m}\le1\), the extended
closed-loop vector field satisfies \(|\widehat F^m(t,x)|\le C(1+|x|)\), for some constant \(C>0\) independent of \(m\). Hence, along any
solution \(x(\cdot)\), we have \(|\dot x(t)|\le C(1+|x(t)|)\) for almost every \(t\in[0,T]\). Grönwall's inequality therefore yields
\[
|x(t)| \le \left(|x_0|+Ct\right)e^{Ct} \le C_T(1+|x_0|),
\qquad t\in[0,T],
\]
for some constant \(C_T>0\) depending only on \(T\) and \(C\).
\end{proof}

The next estimate controls the dependence of trajectories on the initial condition.

\begin{proposition}[Stability of trajectories] \label{prop:trajectory-stability}
Assume that {\rm (H1)--(H4)} and {\rm (E1)--(E2)} hold. Then there exists a constant $L_T>0$ such that
\[
|x_{0,x_0}(t)-x_{0,y_0}(t)| \le L_T|x_0-y_0|, \qquad t\in[0,T],
\]
for all initial states
$x_0,y_0\in\mathbb R_+^2$.
\end{proposition}

\begin{proof}
The trajectories satisfy \(\frac{\diff}{\diff t} \bigl(x_{0,x_0}(t)-x_{0,y_0}(t)\bigr) = \widehat F^m(t,x_{0,x_0}(t)) - \widehat F^m(t,x_{0,y_0}(t))\), and the Lipschitz property gives
\[
\frac{d}{dt} |x_{0,x_0}(t)-x_{0,y_0}(t)| \le L |x_{0,x_0}(t)-x_{0,y_0}(t)|.
\]
Gronwall's inequality yields \(|x_{0,x_0}(t)-x_{0,y_0}(t)| \le e^{Lt}|x_0-y_0|\), by which the result follows with setting $L_T=e^{LT}$.
\end{proof}

The preceding stability result provides the main continuity property of the characteristic flows with respect to the measurable population distribution. For the fixed-point construction, we also need  a priori uniform estimates on the population flows generated by the closed-loop dynamics. In particular, these estimates must hold uniformly with respect to the candidate population flow \(m\) in the class on which the fixed-point operator will be defined. Three properties will be used to obtain the compactness required by Schauder's fixed-point theorem, consist of a uniform bound on the first moments, a uniform control of the tails of the population distributions, and a uniform modulus of continuity with respect to time. The first two properties ensure relative compactness in \(\mathcal P_1(\mathbb R_+^2)\), while the third provides equicontinuity of the corresponding population trajectories in the Wasserstein
metric. We establish these estimates in the following three results.


\begin{lemma}[Uniform first-moment estimate]
\label{lem:uniform-first-moment}
Assume that {\rm (H1)--(H4)} and {\rm (E1)--(E2)} hold. Then there exists a constant $C_T>0$, depending only on $T$, the initial measure $m_0$, and the growth constants of the closed-loop vector field, such that
\[
\sup_{m\in\mathcal C_T}\sup_{t\in[0,T]}
\int_{\mathbb R_+^2} |x|\,(\mathcal T(m))(t,\diff x)
\le C_T.
\]
\end{lemma}

\begin{proof}
Let $m\in\mathcal C_T$ and set $\widetilde m=\mathcal T(m)$. By the definition of the fixed-point operator,
we have \(\widetilde m(t)=(\Phi^m_{0,t})_{\#}m_0\). Hence,
\[
\int_{\mathbb R_+^2}|x|\,\widetilde m(t,\diff x) =
\int_{\mathbb R_+^2}|\Phi^m_{0,t}(x)|\,m_0(\diff x).
\]
By the uniform linear growth estimate for the closed-loop trajectories, stated in Proposition~\ref{prop:trajectory-growth}, there exists a constant $C_T>0$, independent of $m\in\mathcal C_T$, such that
\(|\Phi^m_{0,t}(x)|\le C_T(1+|x|)\), for \(t\in[0,T]\).
Therefore,
\[
\int_{\mathbb R_+^2}|x|\,\widetilde m(t,\diff x) \le
C_T\left(1+\int_{\mathbb R_+^2}|x|\,m_0(\diff x) \right),
\]
and the right-hand side is independent of $m$ and $t$, which proves the claim.
\end{proof}

The next result provides uniform control of the tails of the first moments, in which, for \(r>0\), we denote the closed ball by \(B_r=\{x\in\mathbb R_+^2\suchthat |x|\le r\}\).

\begin{lemma}[Uniform tail estimate]
\label{lem:uniform-tail}
Assume that {\rm (H1)--(H4)} and {\rm (E1)--(E2)} hold. Then there exists a decreasing function
\(\omega_T^{\mathrm{tail}}\colon[0,+\infty)\longrightarrow[0,+\infty)\) with \(\omega_T^{\mathrm{tail}}(R)\longrightarrow 0\) as \(R\longrightarrow+\infty\),
such that
\[
\sup_{m\in\mathcal C_T}\sup_{t\in[0,T]} \int_{\mathbb R_+^2\setminus B_R} |x|\,(\mathcal T(m))(t,\diff x)
\le \omega_T^{\mathrm{tail}}(R), \qquad R>0.
\]
\end{lemma}

\begin{proof}
Let $m\in\mathcal C_T$ and set $\widetilde m=\mathcal T(m)$. Using the push-forward representation,
\[
\int_{\mathbb R_+^2\setminus B_R} |x|\,\widetilde m(t,\diff x) = \int_{\{|\Phi^m_{0,t}(x)|>R\}}
|\Phi^m_{0,t}(x)|\,m_0(\diff x).
\]
By the uniform linear growth estimate, we have\(|\Phi^m_{0,t}(x)|\le C_T(1+|x|)\). Consequently,
\(\{|\Phi^m_{0,t}(x)|>R\} \subset \left\{|x|>\big(\frac{R}{C_T}-1\big)_+ \right\}\), and therefore
\[
\int_{\mathbb R_+^2\setminus B_R} |x|\,\widetilde m(t,\diff x) \le C_T \int_{\{|x|>(R/C_T-1)_+\}} (1+|x|)\,m_0(\diff x).
\]
Since $m_0\in\mathcal P_1(\mathbb R_+^2)$, the right-hand side converges to zero as $R\longrightarrow +\infty$, uniformly with respect to $m\in\mathcal C_T$ and $t\in[0,T]$. Defining, for instance,
\[
\omega_T^{\mathrm{tail}}(R) := C_T \int_{\{|x|>(R/C_T-1)_+\}}
(1+|x|)\,m_0(\diff x),
\]
with a harmless modification for small $R$, gives the desired estimate.
\end{proof}

The next result provides equicontinuity of the population trajectories in the Wasserstein metric.

\begin{lemma}[Uniform temporal equicontinuity]
\label{lem:uniform-equicontinuity}
Assume that {\rm (H1)--(H4)} and {\rm (E1)--(E2)} hold. Then there exists a modulus of continuity
\(\omega_T^{\mathrm{time}}\colon[0,+\infty)\longrightarrow[0,+\infty)\) with \(\omega_T^{\mathrm{time}}(r)\longrightarrow 0\) as \(r\longrightarrow0^+\), such that
\[
\sup_{m\in\mathcal C_T} W_1\bigl((\mathcal T(m))(t),(\mathcal T(m))(s)\bigr) \le \omega_T^{\mathrm{time}}(|t-s|), \quad \forall\, s,t\in[0,T].
\]
\end{lemma}

\begin{proof}
Let $m\in\mathcal C_T$ and set $\widetilde m=\mathcal T(m)$. Since \(\widetilde m(t)=(\Phi^m_{0,t})_{\#}m_0\), the coupling induced by $m_0$ gives
\[
W_1(\widetilde m(t),\widetilde m(s)) \le \int_{\mathbb R_+^2} |\Phi^m_{0,t}(x)-\Phi^m_{0,s}(x)| \,m_0(\diff x).
\]
Using the integral formulation of the characteristics and the uniform linear growth estimate,
\[
|\Phi^m_{0,t}(x)-\Phi^m_{0,s}(x)| \le \int_s^t \left|\widehat F^m(\tau,\Phi^m_{0,\tau}(x))\right|
\,\diff\tau \le C_T(1+|x|)|t-s|.
\]
Hence,
\[
W_1(\widetilde m(t),\widetilde m(s)) \le C_T \left(
1+\int_{\mathbb R_+^2}|x|\,m_0(\diff x) \right)|t-s|.
\]
Therefore, by defining
\[
\omega_T^{\mathrm{time}}(r) := C_T \left(1+\int_{\mathbb R_+^2}|x|\,m_0(\diff x) \right)r,
\]
we obtain the desired result, in which the estimate is uniform with respect to $m\in\mathcal C_T$.
\end{proof}

The preceding estimates identify the three uniform properties that will be imposed on the class of candidate population flows used in the
fixed-point argument. The first-moment estimate provides uniform control of the mass in the first moment, the tail estimate provides uniform integrability, and the temporal estimate provides equicontinuity in the Wasserstein metric. Together, these properties
are precisely those required to obtain compactness in
\(\mathcal C_T\). This motivates the following definition of the invariant class on which the fixed-point argument will be carried out.

\begin{definition}[Compact class of population flows]
\label{def:MT}
Let \(C_T>0\). Let \(\omega_T^{\mathrm{tail}}\colon [0,+\infty)\longrightarrow[0,+\infty)\) satisfy
\(\omega_T^{\mathrm{tail}}(R)\longrightarrow 0\) as \(R\longrightarrow+\infty\), and let
\(\omega_T^{\mathrm{time}}\colon [0,+\infty)\longrightarrow[0,+\infty)\) satisfy
\(\omega_T^{\mathrm{time}}(0)=0\) and \(\omega_T^{\mathrm{time}}(r)\longrightarrow 0\) as \(r\longrightarrow0^+\). We define \(\mathcal M_T\subset\mathcal C_T\) as the set of population flows
\(m\in\mathcal C_T\) satisfying the following properties.
\begin{itemize}
     \item \textbf{Uniform first-moment bound:} 
    \(\sup_{t\in[0,T]} \int_{\mathbb R_+^2}|x|\,m(t,\mathrm dx) \le C_T\).
    
    \item \textbf{Uniform tail estimate:}
    \(\sup_{t\in[0,T]}
    \int_{\mathbb R_+^2\setminus B_R}
    |x|\,m(t,\mathrm dx) \le \omega_T^{\mathrm{tail}}(R)\), for all \(R>0\).
    
    \item \textbf{Uniform temporal continuity:}
    \(W_1(m(t),m(s)) \le \omega_T^{\mathrm{time}}(|t-s|)\), for all \(s,t\in[0,T]\).
\end{itemize}
The first condition provides a finite-horizon uniform bound on the first moments, while the second condition is a finite-horizon uniform integrability condition that guarantees tightness in \(\mathcal P_1(\mathbb R_+^2)\). At the end, the third condition provides equicontinuity of the population flows in the Wasserstein metric.
\end{definition}

\begin{proposition}[Invariance of \(\mathcal M_T\) under \(\mathcal T\)]
\label{prop:MT-invariance}
Assume that {\rm (H1)--(H4)} and {\rm (E1)--(E2)} hold. Then \(\mathcal T(\mathcal C_T)\subseteq\mathcal M_T\). In particular, \(\mathcal T(\mathcal M_T)\subseteq\mathcal M_T\).
\end{proposition}

\begin{proof}
Let \(m\in\mathcal C_T\) and set \(\widetilde m=\mathcal T(m)\). By Lemma~\ref{lem:uniform-first-moment}, we deduce 
\(\sup_{t\in[0,T]} \int_{\mathbb R_+^2}|x|\,\widetilde m(t,\diff x) \le C_T\). By Lemma~\ref{lem:uniform-tail}, we observe that
\[
\sup_{t\in[0,T]} \int_{\mathbb R_+^2\setminus B_R} |x|\,\widetilde m(t,\diff x) \le \omega_T^{\mathrm{tail}}(R), \qquad R>0.
\]
And finally, by Lemma~\ref{lem:uniform-equicontinuity},
\(W_1(\widetilde m(t),\widetilde m(s)) \le \omega_T^{\mathrm{time}}(|t-s|)\), for \(s,t\in[0,T]\). Thus \(\widetilde m\) satisfies all three defining properties of \(\mathcal M_T\), and hence \(\widetilde m\in\mathcal M_T\). Therefore \(\mathcal T(\mathcal C_T)\subseteq\mathcal M_T\), and the invariance of \(\mathcal M_T\) follows immediately.
\end{proof}

We now show the convexity and compactness of \(\mathcal M_T\).

\begin{proposition}[Convexity and compactness of $\mathcal M_T$]
\label{prop:MT-convex-compact}
Assume that {\rm (H1)--(H4)} and {\rm (E1)--(E2)} hold. The set $\mathcal M_T$ is convex and compact in
$\mathcal C_T$ endowed with the metric \(d_T\).
\end{proposition}

\begin{proof}
We first prove convexity. Let \(m^1,m^2\in\mathcal M_T\) and let \(\lambda\in[0,1]\). Define \(m^\lambda(t) = \lambda m^1(t)+(1-\lambda)m^2(t)\).
Since the first moment is linear with respect to the measure,
\[
\int |x|\,m^\lambda(t,\mathrm dx) = \lambda\int |x|\,m^1(t,\mathrm dx) + (1-\lambda)\int |x|\,m^2(t,\mathrm dx),
\]
the first-moment bound is preserved. The tail estimate is also preserved by convexity,
\[
\int_{\mathbb R_+^2\setminus B_R} |x|\,m^\lambda(t,\mathrm dx) \le \lambda\omega_T^{\mathrm{tail}}(R) + (1-\lambda)\omega_T^{\mathrm{tail}}(R) = \omega_T^{\mathrm{tail}}(R).
\]
Finally, the Wasserstein distance is convex with respect to convex combinations of measures. Hence,
\[
W_1(m^\lambda(t),m^\lambda(s)) \le \lambda W_1(m^1(t),m^1(s)) + (1-\lambda) W_1(m^2(t),m^2(s))
\le \omega_T^{\mathrm{time}}(|t-s|).
\]
Therefore, \(m^\lambda\in\mathcal M_T\).

We next prove compactness. We first establish relative compactness. Fix $t\in[0,T]$ and consider the family
\(\mathcal F_t = \bigl\{m(t)\suchthat m\in\mathcal M_T\bigr\} \subset\mathcal P_1(\mathbb R_+^2)\). The uniform first-moment bound, in the definition of \(\mathcal M_T\), implies tightness of $\mathcal F_t$. Indeed, by Markov's inequality,
\[
\sup_{m\in\mathcal M_T} m(t)(\mathbb R_+^2\setminus B_R) \le \frac{1}{R} \sup_{m\in\mathcal M_T} \int_{\mathbb R_+^2}|x|\,m(t,\diff x) \le \frac{C_T}{R},
\]
which tends to zero as $R\longrightarrow +\infty$. Moreover, the uniform tail estimate gives uniform integrability of the first moments
\[
\sup_{m\in\mathcal M_T} \int_{\mathbb R_+^2\setminus B_R} |x|\,m(t,\diff x) \le \omega_T^{\mathrm{tail}}(R)
\longrightarrow 0 \qquad\text{as }R\longrightarrow +\infty.
\]
Consequently, the family $\mathcal F_t$ is relatively compact in $(\mathcal P_1(\mathbb R_+^2),W_1)$. The temporal continuity estimate provides a common modulus of continuity for the family $\mathcal M_T$, \(W_1\bigl(m(t),m(s)\bigr) \le \omega_T^{\mathrm{time}}(|t-s|)\), for \(m\in\mathcal M_T\). Thus, by the Arzel\`a--Ascoli theorem for metric-valued functions, the family $\mathcal M_T$ is relatively compact in $\mathcal C_T$ endowed with $d_T$. It remains to prove that $\mathcal M_T$ is closed. Let $(m_n)_{n\in\mathbb N}\subset\mathcal M_T$ and suppose that \(d_T(m_n,m)\longrightarrow 0\) for some $m\in\mathcal C_T$. Then, for every $t\in[0,T]$, \(m_n(t)\longrightarrow m(t)\) in \((\mathcal P_1(\mathbb R_+^2),W_1)\). By lower semicontinuity of the first moment,
\[
\int_{\mathbb R_+^2}|x|\,m(t,\diff x) \le \liminf_{n\to\infty} \int_{\mathbb R_+^2}|x|\,m_n(t,\diff x) \le C_T.
\]
Likewise, for every \(R>0\), since \(B_R\) is closed, the function \(x\longmapsto |x|\mathbf 1_{\mathbb R_+^2\setminus B_R}(x)\) is nonnegative and lower semicontinuous. Hence, by the Portmanteau theorem (see, e.g.\@\cite[Theorem~2.1]{Billingsley1999}),
\[
\int_{\mathbb R_+^2\setminus B_R}|x|\,m(t,\diff x) \le
\liminf_{n\to\infty} \int_{\mathbb R_+^2\setminus B_R}|x|\,m_n(t,\diff x) \le \omega_T^{\mathrm{tail}}(R).
\]
Finally, since $W_1$ is continuous with respect to both of its arguments,
\[
W_1\bigl(m(t),m(s)\bigr) = \lim_{n\to\infty} W_1\bigl(m_n(t),m_n(s)\bigr) \le \omega_T^{\mathrm{time}}(|t-s|), \quad \text{ for all } s,t\in[0,T].
\]
Therefore $m\in\mathcal M_T$, and $\mathcal M_T$ is closed in $\mathcal C_T$. Since $\mathcal M_T$ is both relatively compact and closed, it is compact in $(\mathcal C_T,d_T)$.
\end{proof}


The fixed-point argument does not require a stability result for the value functions themselves. In particular, since the Hamilton--Jacobi equation is understood in the viscosity sense, we do not assume that the value functions are differentiable, nor do we deduce convergence of their gradients from convergence of the value functions. What is needed for the fixed-point construction is the stability of the selected optimal feedback entering the closed-loop dynamics. We therefore impose the following assumption.
\begin{enumerate}
    \item[(E3)] \textbf{Stability of the selected optimal feedback.}
    Let \(m_n,m\) be measurable population flows such that \(m_n\longrightarrow m\) uniformly in the Wasserstein distance on every compact time interval. Let \(u^{*,m_n}\) and \(u^{*,m}\) denote the selected optimal feedbacks associated with the corresponding individual minimal-time problems. Then \(u^{*,m_n}\longrightarrow u^{*,m}\) locally uniformly on \([0,+\infty)\times(\mathbb R_+^2\setminus\Gamma)\). Moreover, the corresponding extended feedbacks satisfy
    \(\widehat u^{\,m_n}\longrightarrow \widehat u^{\,m}\), locally uniformly on \([0,+\infty)\times\mathbb R_+^2\).
\end{enumerate}

Assumption {\rm (E3)} is a stability assumption on the selected optimal feedback rather than on the value function. It allows the fixed-point operator to be studied without requiring classical differentiability of
the value function or convergence of its spatial gradients.

\begin{remark}
    Assumptions {\rm (E1)--(E3)} are imposed directly on the population flow and on the closed-loop dynamics induced by the individual optimization problem. In particular, {\rm (E2)--(E3)} are conditions on the regularity and stability of the optimal feedback and the associated velocity field, rather than primitive assumptions stated directly in terms of \(K\), \(w\), and \(c\). Such properties can be obtained in more regular settings from additional regularity of the value function together with suitable non-degeneracy and uniqueness conditions for the Hamiltonian minimizer. Establishing general primitive conditions guaranteeing {\rm (E2)--(E3)} is beyond the scope of the present work. The assumptions are satisfied in particular classes of regular models and are adopted here to isolate the fixed-point argument from the separate regularity analysis of the individual control problem.
\end{remark}

\begin{lemma}[Stability of the closed-loop dynamics]
\label{lem:closed-loop-stability}
Assume that {\rm (H1)--(H4)} and {\rm (E1)--(E3)} hold. Let \(m_n\) and \(m\) be measurable population flows such that \(\sup_{t\in[0,T]}W_1(m_n(t),m(t))\longrightarrow 0\). Then the corresponding extended closed-loop vector fields satisfy \(\widehat F^{m_n}\longrightarrow \widehat F^m\), locally uniformly on \([0,T]\times\mathbb R_+^2\). Moreover, if \(\Phi^{m_n}_{s,t}\) and \(\Phi^m_{s,t}\) denote the corresponding characteristic flows, then
\[
\sup_{t\in[0,T]} \left|\Phi^{m_n}_{0,t}(x)-\Phi^m_{0,t}(x)\right| \longrightarrow 0,
\]
locally uniformly with respect to the initial condition \(x\).

\end{lemma}

\begin{proof}
By {\rm (E1)}, one observes that \(\sup_{t\in[0,T]}
\left|K(m_n(t),z)-K(m(t),z)\right| \longrightarrow 0\), locally uniformly with respect to \(z\). Assumption {\rm (E3)} gives the local uniform convergence of the corresponding extended feedbacks.
Together with the continuity of the controlled vector field, this yields \(\widehat F^{m_n}\longrightarrow \widehat F^m\), locally uniformly on \([0,T]\times\mathbb R_+^2\). Since, by {\rm (E2)}, the vector fields are globally Lipschitz with respect
to the state variable with a common Lipschitz constant, the continuous
dependence theorem for Carath\'eodory differential equations yields
\[
\sup_{t\in[0,T]} \left|\Phi^{m_n}_{0,t}(x)-\Phi^m_{0,t}(x)\right| \longrightarrow 0,
\]
uniformly for \(x\) in compact subsets of \(\mathbb R_+^2\).
\end{proof}

\begin{proposition}[Continuity of the fixed-point operator]
\label{prop:T-continuity}
Assume that {\rm (H1)--(H4)} and {\rm (E1)--(E3)} hold. Then
\(\mathcal T\colon \mathcal C_T\longrightarrow\mathcal C_T\) is continuous
with respect to the metric \(d_T\).
\end{proposition}

\begin{proof}
Let \(m_n\longrightarrow m\) in \(\mathcal C_T\). By
Lemma~\ref{lem:closed-loop-stability}, the corresponding characteristic flows satisfy
\[
\sup_{t\in[0,T]} \left|\Phi^{m_n}_{0,t}(x)-\Phi^m_{0,t}(x)\right| \longrightarrow 0,
\]
locally uniformly with respect to the initial condition \(x\in\mathbb R_+^2\). Using the above expression as \(g_n(x):=\sup_{t\in[0,T]} \left|\Phi^{m_n}_{0,t}(x)-\Phi^m_{0,t}(x)\right|\), we have \(g_n(x)\longrightarrow 0\), for every \(x\in\mathbb R_+^2\). By the linear growth estimate from Proposition~\ref{prop:trajectory-growth}, there exists a constant \(C_T>0\), independent of \(n\), such that \(g_n(x) \le C_T(1+|x|)\), for \(x\in\mathbb R_+^2\). Since \(m_0\in\mathcal P_1(\mathbb R_+^2)\), the function
\(x\mapsto C_T(1+|x|)\) is integrable with respect to \(m_0\). Hence, by the dominated convergence theorem, we observe
\[
\int_{\mathbb R_+^2} g_n(x)\,m_0(\mathrm dx) \longrightarrow 0.
\]
On the other hand, since \(\mathcal T(m_n)(t) = \Phi^{m_n}_{0,t}\#m_0\) and \(\mathcal T(m)(t) = \Phi^m_{0,t}\#m_0\), the coupling induced by the common initial distribution \(m_0\) gives, for every \(t\in[0,T]\),
\[
W_1\bigl(\mathcal T(m_n)(t),\mathcal T(m)(t)\bigr)
\le \int_{\mathbb R_+^2} \left| \Phi^{m_n}_{0,t}(x)-\Phi^m_{0,t}(x) \right| \,m_0(\mathrm dx).
\]
Taking the supremum over \(t\in[0,T]\), we obtain
\[
\begin{aligned}
d_T\bigl(\mathcal T(m_n),\mathcal T(m)\bigr)
= \sup_{t\in[0,T]} W_1\bigl(\mathcal T(m_n)(t),\mathcal T(m)(t)\bigr) 
&\le \int_{\mathbb R_+^2} \sup_{t\in[0,T]} \left|
\Phi^{m_n}_{0,t}(x)-\Phi^m_{0,t}(x) \right| \,m_0(\mathrm dx) \\
&= \int_{\mathbb R_+^2} g_n(x)\,m_0(\mathrm dx) \longrightarrow 0.
\end{aligned}
\]
Therefore, \(d_T\bigl(\mathcal T(m_n),\mathcal T(m)\bigr)\longrightarrow 0\), which proves the continuity of \(\mathcal T\).
\end{proof}

Although the relative compactness of \(\mathcal T(\mathcal M_T)\) follows immediately from the compactness of \(\mathcal M_T\) established in Proposition~\ref{prop:MT-convex-compact} and the invariance property \(\mathcal T(\mathcal M_T)\subseteq\mathcal M_T\) established in
Proposition~\ref{prop:MT-invariance}, the preceding uniform estimates stated in Lemmas~\ref{lem:uniform-first-moment}--\ref{lem:uniform-equicontinuity} also provide a direct proof of this property. In particular, we record the following relative compactness result.


\begin{proposition}[Relative compactness of $\mathcal T(\mathcal M_T)$]
\label{prop:relative-compactness}
Assume that {\rm (H1)--(H4)} and {\rm (E1)--(E3)} hold. Then \(\mathcal T(\mathcal M_T)\subset\mathcal C_T\)
is relatively compact with respect to the metric $d_T$.
\end{proposition}

\begin{proof}
By Lemma~\ref{lem:uniform-first-moment}, the family
\(\left\{(\mathcal T(m))(t)\suchthat m\in\mathcal M_T,\ t\in[0,T] \right\}\) has uniformly bounded first moments. By Lemma~\ref{lem:uniform-tail}, its first moments are uniformly integrable. Hence, for every fixed $t\in[0,T]$, the set \(\left\{(\mathcal T(m))(t)\suchthat m\in\mathcal M_T \right\}\) is relatively compact in $(\mathcal P_1(\mathbb R_+^2),W_1)$. Moreover, Lemma~\ref{lem:uniform-equicontinuity} provides a common modulus of continuity in time. Since $(\mathcal P_1(\mathbb R_+^2),W_1)$ is a complete separable metric space, the Arzel\`a--Ascoli theorem for metric-valued functions implies that a family of curves which is pointwise relatively compact and uniformly equicontinuous is relatively compact in \(C([0,T];\mathcal P_1(\mathbb R_+^2))\). Therefore, $\mathcal T(\mathcal M_T)$ is relatively compact in
$(\mathcal C_T,d_T)$.
\end{proof}

The preceding results provide all the ingredients required for the application of Schauder's fixed-point theorem. The set \(\mathcal M_T\) is nonempty, convex, and compact by
Proposition~\ref{prop:MT-convex-compact}, the operator
\(\mathcal T\) is continuous on \(\mathcal M_T\) by
Proposition~\ref{prop:T-continuity}, and
Proposition~\ref{prop:MT-invariance} shows that
\(\mathcal T(\mathcal M_T)\subseteq\mathcal M_T\).
We can therefore apply Schauder's fixed-point theorem to
\(\mathcal T\colon\mathcal M_T\longrightarrow\mathcal M_T\).

\begin{theorem}[Existence of a finite-horizon MFG equilibrium]
\label{thm:existence-mfg}
Assume that {\rm(H1)--(H4)} and {\rm(E1)--(E3)} hold. Let \(T>0\) and let \(m_0\in\mathcal P_1(\mathbb R_+^2)\) be the initial population distribution.
Assume furthermore that the post-qualification feedback
\(\bar u\colon [0,T]\times\Gamma\longrightarrow[0,1]\) is measurable and that the associated closed-loop dynamics are well posed. Then there exists a population flow \(m^\star\in\mathcal C_T\) such that \(\mathcal T(m^\star)=m^\star\). Consequently, \(m^\star\) is an MFG equilibrium on \([0,T]\) in the sense of Definition~\ref{def: MFG equilibrium}. Moreover, if \(V^{m^\star}\) denotes the value function associated with \(m^\star\) and \(u^{*,m^\star}\) the corresponding selected optimal feedback, then \(m^\star\) is generated by the associated optimal closed-loop dynamics on \([0,T]\).
\end{theorem}

\begin{proof}
We verify the hypotheses of Schauder's fixed-point theorem.

\paragraph*{Nonemptiness, convexity, and compactness.}
The set $\mathcal M_T$ is nonempty. Indeed, the constant population flow $m(t)\equiv m_0$ belongs to $\mathcal M_T$ provided that $C_T$ is chosen larger than $\int_{\mathbb R_+^2}|x|\,m_0(\mathrm dx)$ and that $\omega_T^{\mathrm{tail}}$ is chosen so as to dominate the corresponding tail integral of $m_0$. The temporal continuity condition is automatically satisfied since $W_1(m(t),m(s))=0$ for all $s,t\in[0,T]$. Convexity and
compactness then follow from Proposition~\ref{prop:MT-convex-compact}.


\paragraph*{Invariance of $\mathcal M_T$.}
By Proposition~\ref{prop:MT-invariance}, we have \(\mathcal T(\mathcal C_T)\subseteq\mathcal M_T\), and hence, in particular, we deduce that \(\mathcal T(\mathcal M_T)\subseteq\mathcal M_T\).

\paragraph*{Continuity of the fixed-point operator.}
The continuity of \(\mathcal T\) is already established in Proposition~\ref{prop:T-continuity}. Hence, its restriction \(\mathcal T\colon\mathcal M_T\longrightarrow\mathcal M_T\) is continuous with respect to the metric \(d_T\).

\paragraph*{Application of Schauder's theorem.}
All the hypotheses of Schauder's fixed-point theorem are now satisfied. More precisely, \(\mathcal M_T\) is nonempty, convex, and compact, \(\mathcal T\) is continuous, and \(\mathcal T(\mathcal M_T)\subseteq\mathcal M_T\). Therefore, there exists
\(m^\star\in\mathcal M_T\) such that \(\mathcal T(m^\star)=m^\star\). By Definition~\ref{def: MFG equilibrium}, \(m^\star\) is therefore an
MFG equilibrium on \([0,T]\). The corresponding value function \(V^{m^\star}\) and selected optimal feedback \(u^{*,m^\star}\) provide the individual optimality conditions associated with the extended population flow \(\mathfrak E_T m^\star\), while \(m^\star\) is generated by the corresponding optimal closed-loop dynamics on \([0,T]\).
\end{proof}

The previous theorem establishes the existence of a self-consistent population trajectory on every finite time interval \([0,T]\), while the individual minimum-time problem remains formulated on the infinite time horizon. For a given population flow \(m\in\mathcal C_T\), the associated infinite-horizon individual problem is defined using the extension \(\mathfrak E_Tm\in\mathcal C_\infty\) with \((\mathfrak E_Tm)(t)=m(t)\), for \(t\in[0,T]\), as specified earlier. Thus, \(T\) should not be interpreted as a terminal time for the individual optimization problem, but only as the horizon of the population fixed-point construction. The equilibrium population flow is obtained as a fixed point of the nonlinear operator
\begin{equation}\label{fixed-point sequence}
m \longmapsto V^m \longmapsto u^{*,m}
\longmapsto \widehat F^m \longmapsto \mathcal T(m),
\end{equation}
which combines the individual optimization problem with the aggregate population dynamics. The finite-horizon formulation is particularly suitable for the fixed-point argument because it provides uniform moment, tail, and temporal continuity estimates for the population
flows, yielding the compactness required by Schauder's theorem. 
Since the individual optimization problem remains formulated on the infinite time horizon, it is natural to ask whether the finite-time population equilibria can be extended to a global-in-time equilibrium. Such a result, however, does not follow directly from the finite-horizon fixed-point construction. 

Let \(T_n\longrightarrow +\infty\), and for each \(T_n\) let \(m^{T_n}\in\mathcal C_{T_n}\) be a finite-horizon equilibrium. A natural strategy is to extract a subsequence which converges on every compact time interval. The key issue is the availability of estimates which are uniform with respect to the horizon. More precisely, suppose that for every \(S>0\), there exist constants \(C_S>0\) and a modulus \(\omega_S\), independent of \(T\geq S\), such that
\[
\sup_{t\in[0,S]} \int_{\mathbb R_+^2}|x|\,m^T(t,\diff x)
\leq C_S,  \,\, \text{ and }  \,\, W_1\bigl(m^T(t),m^T(s)\bigr)
\leq \omega_S(|t-s|), \,\, \text{ for } s,t\in[0,S].
\]
If, in addition, a uniform tail estimate holds on every compact time interval, then the same compactness argument as above implies that the family \(\{m^T\suchthat T\geq S\}\) is relatively compact in
\(C\bigl([0,S];\mathcal P_1(\mathbb R_+^2)\bigr)\).
Applying this argument successively on the intervals \([0,S]\) and using a diagonal extraction procedure, one may obtain a subsequence, still denoted by \(m^{T_n}\), and a limit measurable distribution \(m\in C_{\mathrm{loc}}\bigl([0,+\infty);
\mathcal P_1(\mathbb R_+^2)\bigr)\) such that, for every finite \(S>0\), \(m^{T_n}\longrightarrow m\) in \(C\bigl([0,S];\mathcal P_1(\mathbb R_+^2)\bigr)\).
The convergence of the population flows alone, however, is not sufficient to pass to the limit in the MFG system. One must also establish suitable stability properties for the individual value functions and the selected optimal feedbacks. In the present framework, the stability of the selected feedback is imposed through assumption {\rm(E3)} for the finite-horizon fixed-point construction. A global limit \(T\longrightarrow +\infty\) would require a corresponding stability property which is uniform with respect to the horizon, together with suitable compactness or stability estimates for the associated value functions.

A further difficulty concerns the passage to the limit in the continuity equation. The velocity field depends on both the population flow and the selected optimal feedback. Consequently, convergence of the population trajectories must be combined with sufficient stability of the individual optimization problem and of the resulting closed-loop dynamics in order to identify the limiting population flow as a solution of the coupled MFG system.
Thus, the construction of a global-in-time MFG equilibrium would require additional uniform-in-horizon estimates, including uniform moment and tail bounds, local compactness of the population flows, and suitable stability properties for the value functions and selected
optimal feedbacks. Establishing these properties is beyond the scope of the present work. The finite-horizon result obtained above therefore constitutes the rigorous existence result of this paper. It establishes the existence of a self-consistent MFG equilibrium on every finite time interval \([0,T]\) and provides a compactness framework that can serve as a starting point for a future analysis of global-in-time equilibria or equivalently the analysis of MFG equilibria with free final time.

\section{Discussion and economic interpretation}
\label{sec: disscussion}

The MFG framework developed in this paper provides a dynamic model of human-capital accumulation in which individual educational decisions interact with the aggregate population distribution. The central feature of the model is that the productivity of learning depends not only on an individual's own state but also on the population environment. As a consequence, educational choices generate external effects that influence the future incentives of other agents. The state of an individual is described by the pair \(x(t)\), and the control variable is the learning effort \(u(t)\in[0,1]\), which determines the intensity of investment in human-capital accumulation. The controlled dynamics are given by~\eqref{eq:control-system}. The first equation describes the accumulation of human capital, in which the learning effort increases the stock of human capital, while the productivity of this effort depends on the learning technology \(K\), which is influenced by the aggregate population distribution. The second equation describes the evolution of financial resources. Income depends on the current level of human capital through the wage function \(w\), while educational effort generates a direct cost represented by the function \(c\).

The individual objective is to reach the qualification threshold \(\Gamma\) in the shortest possible time. The optimization problem is therefore a minimum-time problem with a state constraint \(b(t)\geq0\). The budget constraint plays a central role in the model, reflecting the fact that the educational effort is financially costly, and individuals cannot sustain learning policies that would drive their financial resources below zero. The model captures a fundamental economic trade-off between the speed of human-capital accumulation and the preservation of financial resources. A higher learning effort accelerates qualification, but it also increases the cost of education and may reduce the accumulation of financial resources. Conversely, a lower effort reduces the educational cost and may allow financial resources to accumulate more rapidly. The optimal policy must balance these two opposing forces over time.

A distinctive feature of the model is that learning productivity is endogenous at the aggregate level. The function \(K(m,z)\) depends on the population distribution \(m\), which means that the return to educational effort varies with the composition of the population. For example, a population with a larger concentration of highly educated individuals may generate knowledge spillovers, better access to information, or more productive learning interactions. Conversely, congestion effects or competition for educational resources may reduce the effectiveness of learning. The framework therefore accommodates both positive and negative aggregate interactions through the dependence of \(K\) on \(m\). This endogenous learning productivity is the mechanism that transforms the individual optimal control problem into an MFG. Individual educational decisions affect the evolution of the population distribution, and the resulting population distribution modifies the productivity of future educational decisions. The economy is therefore described by a dynamic feedback loop between individual human-capital accumulation and aggregate educational conditions.


The MFG equilibrium is fundamentally a fixed-point problem linking individual optimization and aggregate population dynamics. Unlike the individual optimal control problem studied in the previous sections, the population environment entering the individual optimization problem is not exogenous at equilibrium. The productivity of learning depends on the population distribution, while the population distribution itself is generated by the optimal learning decisions of the agents. The equilibrium is therefore characterized by a consistency condition between the population environment used by individuals when making their decisions and the population environment generated by their collective behavior.

More precisely, we have considered a candidate population flow \(m\in\mathcal C_T = C\bigl([0,T];\mathcal P_1(\mathbb R_+^2)\bigr)\). The horizon \(T\) refers to the time interval on which the population fixed-point construction is performed. The individual minimum-time problem, however, remains formulated on the infinite time horizon. Accordingly, the candidate population flow \(m\) is incorporated into the individual optimization problem through the extension \(\mathfrak E_Tm\in\mathcal C_\infty\), with \(\mathcal C_\infty = C\bigl([0,+\infty);\mathcal P_1(\mathbb R_+^2)\bigr)\),
which coincides with \(m\) on \([0,T]\) and is taken to be constant, equal to \(m(T)\), after time \(T\). Thus, for a fixed population flow and its extension, each agent solves the infinite-horizon minimum-time optimal control problem introduced in Section~\ref{sec: Human-capital model}. The solution of this problem is described by the value function, which represents the minimal remaining time required for an individual with human capital \(z\) and budget \(b\) at time \(t\) to reach the qualification threshold. The value function summarizes the future optimization possibilities and therefore provides a characterization of optimal individual behavior.

The value function determines the selected optimal feedback control \(u^{*,m}(t,x)\), which specifies the learning effort that minimizes the remaining qualification time in the population environment generated by \(\mathfrak E_Tm\). This feedback induces a velocity field in the state space, denoted by \(\widehat F^m(t,x)\), and therefore determines the evolution of individual trajectories. The aggregate dynamics are described by the continuity equation, which transports the initial population distribution under this optimal velocity field. Consequently, the candidate population flow \(m\) generates a new population trajectory, denoted by \(\mathcal T(m)\). The equilibrium condition is precisely the fixed-point identity \(m=\mathcal T(m)\). This mechanism may be summarized by the sequence~\eqref{fixed-point sequence}, established in Section~\ref{sec:MFG-system}. The first arrow represents the individual optimization problem in a given population environment. The second arrow determines the optimal learning policy. The third arrow constructs the induced aggregate velocity field, and the final arrow transports the population distribution according to this velocity field. The role of the value function is therefore substantially richer than in a standard optimal control problem. In addition to characterizing optimal individual behavior, it determines the aggregate evolution of the population through the induced feedback. The continuity equation plays the complementary role of translating the microscopic decisions of a continuum of agents into the macroscopic evolution of the population distribution.

This distinguishes the present model from a collection of independent optimal control problems. If the learning technology were independent of the population distribution, each agent would solve an optimization problem unaffected by the behavior of the rest of the population. The aggregate distribution would then simply result from the aggregation of individual decisions, without feeding back into those decisions. In the present framework, however, the population distribution enters the learning technology through the term \(K(m(t),z)\). Consequently, the productivity of learning depends on the aggregate state of the population, and individual incentives are affected by the distribution of human capital and financial resources across agents. The equilibrium is thus a self-consistent dynamic configuration in which individual optimization and aggregate evolution interact. Agents optimize in a given population environment, while that population environment evolves as a consequence of their optimal decisions. The fixed-point condition ensures that the population environment used in the individual optimization problem coincides with the population environment generated by the resulting collective behavior. In this sense, the MFG equilibrium provides a mathematical representation of a decentralized economy in which human-capital accumulation is simultaneously an individual decision and a collective phenomenon driven by endogenous interactions across the population.


The MFG framework developed in this paper provides a natural setting for analyzing policies that affect the process of human-capital accumulation. The comparative statics results of Section~\ref{sec: Human-capital model} show that improvements in the learning technology, increases in income, and reductions in the cost of educational effort weakly decrease the minimal time required to reach the qualification threshold. In the MFG, however, these interventions have an additional equilibrium effect, as they modify the population distribution, which in turn changes the learning environment faced by other agents. Consequently, the aggregate impact of a policy cannot be inferred solely from its effect on an individual agent.

A first class of policies concerns improvements in the learning technology \(K\). Such policies may represent investments in educational infrastructure, improvements in teaching quality, digital learning technologies, access to information, or institutional reforms that increase the efficiency of skill acquisition. At the individual level, a larger value of \(K(m,z)\) increases the rate of human-capital accumulation generated by a given learning effort. As established by the comparative statics results, this weakly decreases the time required to attain the qualification threshold. In equilibrium, however, the effect may be amplified. Faster human-capital accumulation modifies the population distribution, potentially increasing the mass of qualified or highly skilled individuals. If the learning technology exhibits positive dependence on the population distribution, this change may further increase learning productivity for other agents. The initial policy intervention may therefore generate a dynamic amplification mechanism through endogenous aggregate spillovers. Conversely, if learning is subject to congestion effects, the aggregate response may be weaker than the individual response.

A second class of policies consists of increasing the income generated by human capital. Such interventions may represent wage subsidies, educational stipends, income support during training, or labor-market policies that increase the returns to acquired skills. In the model, a higher income function \(w(z)\) increases the accumulation of financial resources and can enlarge the set of feasible learning strategies. Individuals are less likely to interrupt or reduce their educational effort because of financial limitations. The equilibrium consequences are again broader than the individual effects. Higher income may enable a larger fraction of the population to engage in intensive learning, thereby accelerating aggregate human-capital accumulation. The resulting change in the population distribution feeds back into the learning technology through \(K(m,z)\). Consequently, income-support policies may generate indirect productivity effects that are absent from a partial-equilibrium analysis.

A third class of interventions reduces the cost of learning effort \(c(u)\). This may correspond to tuition reductions, subsidized training programs, public financing of education, or policies that lower the opportunity cost of acquiring skills. A reduction in \(c(u)\) makes intensive learning more attractive and enlarges the set of feasible trajectories for financially constrained individuals. The value function consequently decreases because agents can sustain higher learning effort without exhausting their available budget.

The interaction between these three policy channels is particularly important. Improvements in learning technology, increases in income, and reductions in educational costs all accelerate human-capital accumulation, but they operate through different mechanisms. Technological improvements affect the productivity of effort, income policies affect the accumulation of financial resources, and cost reductions affect the budgetary burden of education. In equilibrium, these mechanisms interact through the endogenous evolution of the population distribution.

The dependence of \(K\) on the population distribution is therefore central for policy evaluation. If the learning technology exhibits positive spillovers from aggregate human capital, educational policies may generate aggregate benefits that exceed the direct benefits accruing to individual agents. Individual learning decisions generate benefits for other agents by improving the learning environment itself. In this case, decentralized individual incentives may therefore differ from the socially efficient allocation, potentially providing a rationale for public intervention. On the other hand, if learning is subject to congestion effects, competition for educational resources, or limited training capacity, the aggregate benefits of educational expansion may be smaller than the sum of individual benefits. Policies that appear highly effective in partial equilibrium may therefore produce more moderate aggregate improvements once the endogenous adjustment of the population distribution is taken into account.

This distinction between partial-equilibrium and equilibrium effects is one of the main contributions of the MFG approach. In a standard optimal control framework, the population environment is fixed, and policy evaluation concerns only the response of an individual agent. In the present framework, the population distribution is endogenous. A policy affects individual incentives, which modify optimal learning trajectories, which alter the population distribution, which then changes the learning technology faced by other agents. The equilibrium response is therefore the outcome of a dynamic feedback loop rather than a simple aggregation of individual responses. From a policy perspective, this implies that interventions aimed at accelerating qualification and human-capital accumulation should be evaluated not only by their direct effects on individual behavior but also by their indirect effects on the aggregate learning environment. The model suggests that educational and labor-market policies may generate important dynamic spillovers through the endogenous evolution of the population distribution. These spillovers may substantially amplify or attenuate the effectiveness of public interventions, depending on the nature of the interaction encoded in the learning technology \(K\). The MFG framework therefore provides a natural tool for distinguishing private incentives from aggregate outcomes and for analyzing the consequences of policies that affect human-capital accumulation in an interacting population.


The analysis developed in this paper provides a mathematical framework for studying human-capital accumulation under financial constraints and aggregate learning interactions. The main theoretical contribution is the formulation of the individual optimization problem as a time-dependent minimum-time control problem and the establishment of the existence of an MFG equilibrium on arbitrary finite time horizons. Here, \(T\) is the horizon of the population fixed-point construction, whereas the individual minimum-time problem remains formulated on the infinite time horizon through the extension \(\mathfrak E_Tm\).

Several important questions remain open, however, and suggest natural directions for future research. The principal limitation of the present analysis concerns the passage from finite population horizons to a global-in-time MFG equilibrium. The existence theorem established in Section~\ref{sec:MFG-system} is formulated on an arbitrary finite interval \([0,T]\), where compactness arguments and uniform estimates on the population dynamics can be exploited. This finite interval does not constitute a terminal time for the individual optimization problem, which remains infinite-horizon through the extension \(\mathfrak E_Tm\). Extending the population equilibrium construction to the infinite horizon would require substantially stronger a priori estimates and a compactness argument as \(T\longrightarrow+\infty\). In particular, one would need estimates on the population distributions that are uniform with respect to the population horizon on every fixed compact time interval, together with appropriate stability properties for the value functions and the associated optimal feedbacks on increasingly long horizons. The discussion at the end of Section~\ref{sec:MFG-system} outlines a possible compactness approach based on a sequence of finite-horizon equilibria, but a complete global-in-time existence theory remains an open problem.

A second issue concerns the treatment of agents after qualification. The conservative formulation adopted in this work retains qualified agents in the population and specifies their subsequent behavior through an exogenous post-qualification feedback. An alternative formulation would treat the qualification set as an absorbing boundary and remove agents from the active learning population once they reach the target. This would lead to a transport equation with absorption and a decreasing total mass outside the qualification set. The two formulations correspond to different economic interpretations and may lead to different equilibrium dynamics. 

The model could also be extended in several directions. A natural extension consists of introducing stochastic shocks in the human-capital dynamics or in the income process. Random learning outcomes, uncertain labor-market conditions, or idiosyncratic income fluctuations would transform the individual optimization problem into a stochastic control problem and the Hamilton--Jacobi equation into a second-order Hamilton--Jacobi--Bellman equation. The corresponding MFG would involve a Fokker--Planck equation instead of a purely deterministic continuity equation. Another important extension concerns the determination of wages. In the present model, the income function \(w(z)\) is prescribed and does not depend on the population distribution. A richer economic framework would allow wages to depend on the population distribution, reflecting labor-market equilibrium, skill premia, or endogenous demand for qualified workers. In that case, aggregate human-capital accumulation would affect not only the productivity of learning through \(K(m,z)\), but also the returns to education through an endogenous wage function. This would create an additional equilibrium feedback channel between educational decisions and labor-market outcomes.

The interaction kernel itself may also be generalized. The function \(K(m,z)\) has been treated abstractly, but one may consider local interactions, network effects, peer-group influences, or spatially structured populations. Such extensions could generate multiple equilibria, path dependence, or persistent inequality in human-capital accumulation, depending on the strength and nature of the aggregate interactions. Finally, the framework is well suited for quantitative applications. The model could be calibrated using data on educational attainment, training costs, income dynamics, and transitions between skill levels. The value function provides a measure of the time required to attain qualification, while the equilibrium population dynamics describe the evolution of the distribution of human capital and financial resources. This makes the model potentially useful for evaluating educational subsidies, income-support policies, training programs, and other interventions aimed at accelerating human-capital accumulation.

These extensions would considerably enrich the economic content of the model while preserving the central mathematical structure developed in this paper. The combination of optimal control, Hamilton--Jacobi theory, transport equations, and fixed-point methods provides a flexible framework for analyzing the interaction between individual educational decisions and aggregate population dynamics. The present work therefore constitutes a first step toward a broader class of MFG models of human-capital accumulation with endogenous learning environments, financial constraints, and heterogeneous populations.

\bibliographystyle{abbrv}
\bibliography{bib}

@book{Becker1964Human,
  author = {Becker, Gary S.},
  title = {Human Capital: A Theoretical and Empirical Analysis, with Special Reference to Education},
  publisher = {The University of Chicago Press},
  address = {Chicago},
  year = {1964},
  edition = {3}
}

@article{BenPorath1967Production,
  author = {Ben-Porath, Yoram},
  title = {The Production of Human Capital and the Life Cycle of Earnings},
  journal = {Journal of Political Economy},
  volume = {75},
  number = {4},
  pages = {352--365},
  year = {1967},
  doi = {10.1086/259291}
}

@article{Lucas1988Mechanics,
  author = {Lucas, Robert E.},
  title = {On the Mechanics of Economic Development},
  journal = {Journal of Monetary Economics},
  volume = {22},
  number = {1},
  pages = {3--42},
  year = {1988},
  doi = {10.1016/0304-3932(88)90168-7}
}

@article{Romer1990Endogenous,
  author = {Romer, Paul M.},
  title = {Endogenous Technological Change},
  journal = {Journal of Political Economy},
  volume = {98},
  number = {5},
  pages = {S71--S102},
  year = {1990},
  doi = {10.1086/261725}
}

@article{BenhabibSpiegel1994Role,
  author = {Benhabib, Jess and Spiegel, Mark M.},
  title = {The Role of Human Capital in Economic Development: Evidence from Aggregate Cross-Country Data},
  journal = {Journal of Monetary Economics},
  volume = {34},
  number = {2},
  pages = {143--173},
  year = {1994},
  doi = {10.1016/0304-3932(94)90047-7}
}

@article{Moretti2004Workers,
  author = {Moretti, Enrico},
  title = {Workers' Education, Spillovers, and Productivity: Evidence from Plant-Level Production Functions},
  journal = {American Economic Review},
  volume = {94},
  number = {3},
  pages = {656--690},
  year = {2004},
  doi = {10.1257/0002828041464623}
}

@techreport{WinstonZimmerman2003Peer,
  author      = {Winston, Gordon C. and Zimmerman, David J.},
  title       = {Peer Effects in Higher Education},
  institution = {National Bureau of Economic Research},
  type        = {NBER Working Paper},
  number      = {9501},
  year        = {2003},
  month       = feb,
  doi         = {10.3386/w9501},
  url         = {https://doi.org/10.3386/w9501}
}

@article{Lyle2009Effects,
  author = {Lyle, David S.},
  title = {The Effects of Peer Heterogeneity on the Production of Human Capital at West Point},
  journal = {American Economic Journal: Applied Economics},
  volume = {1},
  number = {1},
  pages = {69--84},
  year = {2009},
  doi = {10.1257/app.1.1.69}
}

@article{HidalgoHidalgo2011Optimal,
  author = {Hidalgo-Hidalgo, Manuel},
  title = {The Optimal Investment in Education in the Presence of Peer Effects},
  journal = {Journal of Economic Dynamics and Control},
  volume = {35},
  number = {9},
  pages = {1585--1599},
  year = {2011}
}

@article{FeldZolitz2017Understanding,
  author = {Feld, Jan and Z{\"o}litz, Ulf},
  title = {Understanding Peer Effects: On the Nature, Identification, and Potential of Peer Effects},
  journal = {Journal of Labor Economics},
  volume = {35},
  number = {2},
  pages = {387--428},
  year = {2017},
  doi = {10.1086/689472}
}

@article{WuZhangWang2023,
  title={Student Performance, Peer Effects, and Friend Networks: Evidence from a Randomized Peer Intervention},
  author={Wu, Jia and Zhang, Junsen and Wang, Chunchao},
  journal={American Economic Journal: Economic Policy},
  volume={15},
  number={1},
  pages={510--542},
  year={2023},
  doi={10.1257/pol.20200563}
}

@article{LasryLionen,
  author = {Lasry, Jean-Michel and Lions, Pierre-Louis},
  title = {Jeux {\`a} champ moyen. I -- Le cas stationnaire},
  journal = {Comptes Rendus Math{\'e}matique},
  volume = {343},
  number = {9},
  pages = {619--625},
  year = {2006},
  doi = {10.1016/j.crma.2006.09.019}
}

@article{LasryLionsfr1,
  author = {Lasry, Jean-Michel and Lions, Pierre-Louis},
  title = {Jeux {\`a} champ moyen. II -- Horizon fini et contr{\^o}le optimal},
  journal = {Comptes Rendus Math{\'e}matique},
  volume = {343},
  number = {10},
  pages = {679--684},
  year = {2006},
  doi = {10.1016/j.crma.2006.09.018}
}

@article{LasryLionsfr2,
  author = {Lasry, Jean-Michel and Lions, Pierre-Louis},
  title = {Jeux {\`a} champ moyen. III -- Identification de la solution d'un probl{\`e}me de Nash},
  journal = {Comptes Rendus Math{\'e}matique},
  volume = {343},
  number = {8},
  pages = {477--482},
  year = {2006},
  doi = {10.1016/j.crma.2006.08.004}
}

@article{HuangMalhame,
  author  = {Huang, Minyi and Malhamé, Roland P. and Caines, Peter E.},
  title   = {Large Population Stochastic Dynamic Games: Closed-Loop McKean--Vlasov Systems and the Nash Certainty Equivalence Principle},
  journal = {IEEE Transactions on Automatic Control},
  volume  = {52},
  number  = {6},
  pages   = {1026--1041},
  year    = {2007},
  doi     = {10.1109/TAC.2007.899849}
}

@article{2HuangMalhame,
  author = {Huang, Minyi and Malham{\'e}, Roland P. and Caines, Peter E.},
  title = {Nash Equilibria for Large-Population Linear-Quadratic-Gaussian Games with Weakly Coupled Agents},
  journal = {Proceedings of the 17th International Symposium on Mathematical Theory of Networks and Systems},
  year = {2006}
}

@article{Huang2003Individual,
  author = {Huang, Minyi and Caines, Peter E. and Malham{\'e}, Roland P.},
  title = {Individual and Mass Behaviour in Large Population Stochastic Wireless Power Control Problems: Centralized and Nash Equilibrium Solutions},
  journal = {Proceedings of the 42nd IEEE Conference on Decision and Control},
  pages = {Individual},
  year = {2003}
}

@article{Aumann1964Markets,
  author = {Aumann, Robert J.},
  title = {Markets with a Continuum of Traders},
  journal = {Econometrica},
  volume = {32},
  number = {1/2},
  pages = {39--50},
  year = {1964},
  doi = {10.2307/1913732}
}

@article{Aumann1974Values,
  author = {Aumann, Robert J.},
  title = {Subjectivity and Correlation in Randomized Strategies},
  journal = {Journal of Mathematical Economics},
  volume = {1},
  number = {1},
  pages = {67--96},
  year = {1974}
}

@article{Jovanovic1988Anonymous,
  author = {Jovanovic, Boyan and Rosenthal, Robert W.},
  title = {Anonymous Sequential Games},
  journal = {Journal of Mathematical Economics},
  volume = {17},
  number = {1},
  pages = {77--87},
  year = {1988}
}

@article{Achdou2020Mean,
  author = {Achdou, Yves and Cardaliaguet, Pierre and Delarue, Fran{\c{c}}ois and Lasry, Jean-Michel and Lions, Pierre-Louis},
  title = {Mean Field Games: Numerical Methods},
  journal = {Oxford Research Encyclopedia of Economics and Finance},
  year = {2020}
}

@book{Carmona2018ProbabilisticI,
  author = {Carmona, Ren{\'e} and Delarue, Fran{\c{c}}ois},
  title = {Probabilistic Theory of Mean Field Games with Applications I: Mean Field FBSDEs, Control, and Games},
  publisher = {Springer},
  series = {Probability Theory and Stochastic Modelling},
  volume = {83},
  address = {Cham},
  year = {2018},
  doi = {10.1007/978-3-319-58920-6}
}

@book{Carmona2018ProbabilisticII,
  author = {Carmona, Ren{\'e} and Delarue, Fran{\c{c}}ois},
  title = {Probabilistic Theory of Mean Field Games with Applications II: Mean Field Games with Common Noise and Master Equations},
  publisher = {Springer},
  series = {Probability Theory and Stochastic Modelling},
  volume = {84},
  address = {Cham},
  year = {2018},
  doi = {10.1007/978-3-319-58920-7}
}

@misc{CardaliaguetNotes,
  author = {Cardaliaguet, Pierre},
  title = {Notes on Mean Field Games},
  note = {From P.-L. Lions' lectures at the Coll{\`e}ge de France},
  year = {2010}
}

@book{Cardaliaguet2019Master,
  author = {Cardaliaguet, Pierre and Delarue, Fran{\c{c}}ois and Lasry, Jean-Michel and Lions, Pierre-Louis},
  title = {The Master Equation and the Convergence Problem in Mean Field Games},
  publisher = {Princeton University Press},
  series = {Annals of Mathematics Studies},
  volume = {201},
  address = {Princeton},
  year = {2019}
}

@article{Gomes2016Regularity,
  author = {Gomes, Diogo A. and Pimentel, Edgard and Voskanyan, Vardan},
  title = {Regularity Theory for Mean Field Game Systems},
  journal = {Springer},
  year = {2016}
}

@misc{LasryLionsGrowth,
  author = {Lasry, Jean-Michel and Lions, Pierre-Louis and Guéant, Olivier},
  title  = {Application of Mean Field Games to Growth Theory},
  year   = {2008}
}

@article{AchdouBueraLasryLionsMoll,
  author = {Achdou, Yves and Buera, Francisco J. and Lasry, Jean-Michel and Lions, Pierre-Louis and Moll, Benjamin},
  title = {Partial Differential Equation Models in Macroeconomics},
  journal = {Philosophical Transactions of the Royal Society A},
  volume = {372},
  number = {2028},
  year = {2014},
  doi = {10.1098/rsta.2013.0580}
}

@article{CarmonaEconomicApplications,
  title={Applications of Mean Field Games in Financial Engineering and Economic Theory},
  author={Carmona, Ren{\'e}},
  journal = {arXiv preprint arXiv:2012.05237},
  year={2020},
  eprint={2012.05237},
  archivePrefix={arXiv},
  primaryClass={q-fin.EC}
}

@article{GhilliRicciZancoHumanCapital,
  title={A Mean Field Game Model for COVID-19 with Human Capital Accumulation},
  author={Ghilli, Daria and Ricci, Cristiano and Zanco, Giovanni},
  journal={Economic Theory},
  volume={77},
  number={1--2},
  pages={533--560},
  year={2024},
  doi={10.1007/s00199-023-01505-0}
}

@book{Aubin1991Viability,
  author = {Aubin, Jean-Pierre},
  title = {Viability Theory},
  publisher = {Birkh{\"a}user},
  address = {Boston},
  year = {1991},
  doi = {10.1007/978-0-8176-4910-4}
}

@book{AubinFrankowska1990SetValued,
  title={Set-Valued Analysis},
  author={Aubin, Jean-Pierre and Frankowska, Halina},
  publisher={Birkh{\"a}user},
  address={Boston},
  year={1990},
  doi={10.1007/978-0-8176-4848-0}
}

@book{BardiCapuzzo,
  author = {Bardi, Martino and Capuzzo-Dolcetta, Italo},
  title = {Optimal Control and Viscosity Solutions of Hamilton-Jacobi-Bellman Equations},
  publisher = {Birkh{\"a}user},
  address = {Boston},
  year = {1997},
  doi = {10.1007/978-0-8176-4755-1}
}

@article{Soner1986Optimal,
  author = {Soner, Halil Mete},
  title = {Optimal Control with State-Space Constraint},
  journal = {SIAM Journal on Control and Optimization},
  volume = {24},
  number = {3},
  pages = {552--561},
  year = {1986},
  doi = {10.1137/0324032}
}

@article{CannarsaFrankowska,
  author = {Cannarsa, Piermarco and Frankowska, H{\'e}l{\`e}ne},
  title = {Some Characterizations of Optimal Trajectories in Control Theory},
  journal = {SIAM Journal on Control and Optimization},
  volume = {29},
  number = {6},
  pages = {1322--1347},
  year = {1991},
  doi = {10.1137/0329068}
}

@article{Mazanti2019Minimal,
  author = {Mazanti, Guilherme and Santambrogio, Filippo},
  title = {Minimal-time mean field games},
  journal = {Math. Models Methods Appl. Sci.},
  fjournal = {Mathematical Models and Methods in Applied Sciences},
  volume = {29},
  number = {8},
  pages = {1413--1464},
  issn = {0218-2025},
  year = {2019},
  doi = {10.1142/S0218202519500258}
}

@article{Sadeghi_Arjmand_2022,
   title={Nonsmooth mean field games with state constraints},
   volume={28},
   ISSN={1262-3377},
   url={http://dx.doi.org/10.1051/cocv/2022069},
   DOI={10.1051/cocv/2022069},
   journal={ESAIM: Control, Optimisation and Calculus of Variations},
   publisher={EDP Sciences},
   author={Sadeghi Arjmand, Saeed and Mazanti, Guilherme},
   year={2022},
   pages={74} }

@article{Arjmand_2022,
  title     = {Multipopulation Minimal-Time Mean Field Games},
  author    = {Sadeghi Arjmand, Saeed and Mazanti, Guilherme},
  journal   = {SIAM Journal on Control and Optimization},
  volume    = {60},
  number    = {4},
  pages     = {1942--1969},
  year      = {2022},
  month     = jul,
  publisher = {Society for Industrial \& Applied Mathematics},
  doi       = {10.1137/21M1407306},
  url       = {https://doi.org/10.1137/21M1407306},
  issn      = {1095-7138}
}

@book{Zeidler1986,
  author    = {Zeidler, Eberhard},
  title     = {Nonlinear Functional Analysis and its Applications I: Fixed-Point Theorems},
  publisher = {Springer},
  address   = {New York},
  year      = {1986}
}

@book{CoddingtonLevinson1955,
  author    = {Earl A. Coddington and Norman Levinson},
  title     = {Theory of Ordinary Differential Equations},
  publisher = {McGraw--Hill},
  address   = {New York},
  year      = {1955}
}

@book{AmbrosioGigliSavare2008,
  author    = {Ambrosio, Luigi and Gigli, Nicola and Savar{\'e}, Giuseppe},
  title     = {Gradient Flows: In Metric Spaces and in the Space of Probability Measures},
  series    = {Lectures in Mathematics ETH Z{\"u}rich},
  edition   = {2},
  publisher = {Birkh{\"a}user},
  address   = {Basel},
  year      = {2008},
  doi       = {10.1007/978-3-7643-8722-8}
}

@book{Rudin1991,
  author    = {Rudin, Walter},
  title     = {Functional Analysis},
  edition   = {2},
  series    = {International Series in Pure and Applied Mathematics},
  publisher = {McGraw-Hill, Inc.},
  address   = {New York},
  year      = {1991},
  isbn      = {0-07-054236-8}
}

@article{Arjmand_2026,
   title={Dynamic resource allocation in eukaryotic {R}esource {B}alance {A}nalysis},
   volume={93},
   ISSN={1432-1416},
   url={http://dx.doi.org/10.1007/s00285-026-02436-9},
   DOI={10.1007/s00285-026-02436-9},
   number={1},
   journal={Journal of Mathematical Biology},
   publisher={Springer Science and Business Media LLC},
   author={Arjmand, Saeed Sadeghi},
   year={2026},
}

@book{Billingsley1999,
  author    = {Billingsley, Patrick},
  title     = {Convergence of Probability Measures},
  edition   = {2},
  series    = {Wiley Series in Probability and Statistics},
  publisher = {John Wiley \& Sons},
  address   = {New York},
  year      = {1999},
  doi       = {10.1002/9780470316962},
  isbn      = {978-0-471-19745-4}
}

\end{document}